\documentclass[11pt]{article}
\usepackage[margin=1in]{geometry}
\usepackage[T1]{fontenc}
\usepackage{lmodern,microtype}
\usepackage{amsmath,amssymb,amsthm,mathtools}
\usepackage{enumitem,booktabs,longtable,array}
\usepackage{aliascnt}
\usepackage[hidelinks]{hyperref}
\usepackage[nameinlink,capitalise]{cleveref}
\hypersetup{pdftitle={Cutoff for the Overhand Shuffle},pdfauthor={Yunjiang Jiang},pdfsubject={Revised manuscript; AI assistance by GPT-6 Astra}}
\numberwithin{equation}{section}
\newtheorem{theorem}{Theorem}[section]
\newaliascnt{lemma}{theorem}\newtheorem{lemma}[lemma]{Lemma}\aliascntresetthe{lemma}
\newaliascnt{proposition}{theorem}\newtheorem{proposition}[proposition]{Proposition}\aliascntresetthe{proposition}
\newaliascnt{corollary}{theorem}\newtheorem{corollary}[corollary]{Corollary}\aliascntresetthe{corollary}
\theoremstyle{definition}
\newaliascnt{definition}{theorem}\aliascntresetthe{definition}
\newaliascnt{remark}{theorem}\aliascntresetthe{remark}
\crefname{lemma}{lemma}{lemmas}\Crefname{lemma}{Lemma}{Lemmas}
\crefname{proposition}{proposition}{propositions}\Crefname{proposition}{Proposition}{Propositions}
\crefname{corollary}{corollary}{corollaries}\Crefname{corollary}{Corollary}{Corollaries}
\crefname{remark}{remark}{remarks}\Crefname{remark}{Remark}{Remarks}
\newcommand{\Sn}{\mathfrak S_n}
\newcommand{\TV}{\mathrm{TV}}
\newcommand{\Ent}{\operatorname{Ent}}
\newcommand{\Var}{\operatorname{Var}}
\newcommand{\Cov}{\operatorname{Cov}}
\newcommand{\E}{\mathbb E}
\newcommand{\Pp}{\mathbb P}
\newcommand{\1}{\mathbf 1}
\newcommand{\eps}{\varepsilon}
\newcommand{\cD}{\mathcal D_n}
\newcommand{\cT}{\mathcal T}
\newcommand{\cE}{\mathcal E}
\newcommand{\cB}{\mathcal B}
\newcommand{\ip}[2]{\langle #1,#2\rangle}
\newcommand{\norm}[1]{\lVert #1\rVert}

\newcommand{\id}{\mathrm{id}}
\title{Cutoff for the Overhand Shuffle}
\author{Yunjiang Jiang\thanks{AI assistance: GPT-6 Astra, used in developing arguments, checking calculations, and preparing the exposition. The author is responsible for the mathematical content.}}
\date{September 2026\quad (revised manuscript)}
\begin{document}
\maketitle
\begin{abstract}
We study the overhand shuffle with a fixed cut probability $p\in(0,1)$. In the packet-reversal convention, each gap is cut independently with probability $p$, and each resulting packet is reversed in place. We prove total-variation cutoff at
\[
 \frac{p^2}{2(1-p)\pi^2}\,n^2\log n.
\]
The lower bound uses an exact finite-deck cosine eigenfunction and a uniform conditional-variance estimate. For the upper bound, a coalescing coupling reduces a transposition discrepancy to a symmetric killed two-card chain. Its principal decay rate is $(1-p)\pi^2/(p^2n^2)+o(n^{-2})$. A boundary-flux identity bounds the sum of adjacent discrepancy probabilities without a factor equal to the deck size. The remaining ingredient is a two-sided $L^1$ inequality on permutations, with constant at most $\sqrt n\exp\{O(\sqrt{\log n\log\log(n+2)})\}$. Its proof decomposes a fixed block-table fiber into assignment words and independent smaller permutations, and closes a recursion using a multivariate hypergeometric inequality and an explicit routing comparison.

We also prove relative-entropy cutoff with an $O_p(n^2)$ window about the intrinsic entropy transition time. Its leading coefficient lies between the total-variation coefficient and twice that value; a sharp entropy location is not obtained. The proof states the normalization and hypotheses of its established functional-inequality inputs. An appendix proves the intermediate static bounds and explains the change from one-scale averaging to recursive fibers.
\end{abstract}
\tableofcontents

\section{Introduction and main results}\label{sec:intro}
\subsection{Historical setting}
Card shuffling gives concrete examples of a basic question in finite probability: how long does a stochastic dynamics take to forget a prescribed initial state? A mixing-time order describes the scale of this loss of information. A cutoff theorem says more: for any fixed proportionate change in time about a specified center, the total-variation distance changes asymptotically from one to zero. Such a statement distinguishes the disappearance of detectable global information from the equilibration of a single card.

Aldous and Diaconis discussed probabilistic models of physical shuffling in \cite{AldousDiaconis1986}. Pemantle's analysis of the overhand shuffle established a lower bound of order $n^2$ and an upper bound of order $n^2\log n$ \cite{Pemantle1989}. Jonasson used an extension of Wilson's eigenfunction method to prove the matching lower order, giving $\Theta_p(n^2\log n)$ mixing for fixed $p$ \cite{Jonasson2006,Wilson2004,Wilson2003}. Diaconis and Pal explicitly listed overhand cutoff as an open problem in their 2022 discussion \cite{DiaconisPal2022}. The result here sharpens the order-of-magnitude statement by identifying a leading coefficient and a total-variation transition.

For comparison, Diaconis--Shahshahani established the classical random-transposition mixing asymptotics \cite{DiaconisShahshahani1981}. Bayer--Diaconis identified the $(3/2)\log_2 n$ total-variation scale for Gilbert--Shannon--Reeds riffle shuffling \cite{BayerDiaconis1992}; Sellke proved asymmetric-riffle cutoff \cite{Sellke2022}, and Sellke--Shi--Wang study general cut distributions \cite{SellkeShiWang2025}. For local dynamics, Lacoin proved cutoff for random adjacent transpositions and simple exclusion \cite{Lacoin2016}, and Nam--Nestoridi treated cyclic systematic scans \cite{NamNestoridi2019}. Bernstein--Nestoridi established random-to-random insertion cutoff using detailed spectral information \cite{BernsteinNestoridi2019}. These are different models: their time normalizations and constants are not imported into the argument below.

Our upper bound instead connects a dynamical coupling to a static inequality involving both positions and labels. Complete-transposition functional inequalities enter as auxiliary tools, through the interchange-process spectral-gap theorem \cite{CaputoLiggettRichthammer2010} and Salez's multislice log-Sobolev theorem \cite{Salez2021}. The entropy argument uses the general Wasserstein-to-entropy theorem of Caputo--M\"unch--Salez \cite{CaputoMunchSalez2025}. Nash estimates for finite chains provide the analytic setting for the two-card smoothing argument \cite{DiaconisSaloffCoste1996}; we prove the finite-grid instance needed here.

\subsection{Model, distances, and statements}
We use function composition: $(ab)(i)=a(b(i))$. A deck $\sigma\in\Sn$ is written with $\sigma(i)$ equal to the label in position $i$. Independently cut each of the $n-1$ gaps with probability $p\in(0,1)$ and reverse each packet. If $R_c$ is the resulting position permutation, the step is $\sigma\mapsto\sigma R_c$. Let $P_n=P_{n,p}$ be this discrete-time kernel and $U_n$ the uniform law on $\Sn$. Every shuffle in this paper is one complete cut-and-reverse pass, not a single packet operation.

All logarithms are natural and times are nonnegative integers. For a finite probability measure $\pi$, let $\norm{f}_{L^a(\pi)}=(\sum_x\pi(x)|f(x)|^a)^{1/a}$, with the usual supremum convention at $a=\infty$. Total variation and relative entropy are
\[
 \norm{\nu-\pi}_{\TV}=\frac12\sum_x|\nu(x)-\pi(x)|,
 \qquad H(\nu\mid\pi)=\sum_x\nu(x)\log\frac{\nu(x)}{\pi(x)},
\]
where $0\log0=0$ and $\pi$ has full support. Write
\[
 \mu_t=\delta_{\id}P_n^t,\quad d_n(t)=\norm{\mu_t-U_n}_{\TV},\quad
 t_{\rm mix}^{(n)}(\eps)=\min\{t:d_n(t)\le\eps\}.
\]
The distances are the same from every deterministic starting deck, by relabeling. Throughout the dynamical argument,
\[
 q=1-p,\qquad D=\frac{q}{p^2},\qquad \kappa=D\pi^2,
 \qquad t_n^\star=\frac{n^2\log n}{2\kappa}.
\]
Constants denoted $C,c$ are positive universal constants whose values may change between displays. A subscript $p$ allows dependence on the fixed parameter $p$, never on $n$. Asymptotic assertions involving the shuffle always keep $p$ fixed.

\begin{theorem}[Total-variation cutoff]\label{thm:TV-main}
For every fixed $p\in(0,1)$ and $\eps\in(0,1)$,
\[
 t_{\rm mix}^{(n)}(\eps)=\left(\frac{p^2}{2(1-p)\pi^2}+o(1)\right)n^2\log n.
\]
More explicitly, for $0<\delta<1$ the distance at $\lfloor(1-\delta)t_n^\star\rfloor$ tends to one, whereas for every $\delta>0$ the distance at $\lceil(1+\delta)t_n^\star\rceil$ tends to zero. At $p=1/2$, the location is $n^2\log n/(4\pi^2)$.
\end{theorem}

Let $s_i=(i,i+1)$, and define
\begin{align}
 B_L(f)&=\sum_{i=1}^{n-1}\E_{U_n}|f(s_i\sigma)-f(\sigma)|,\nonumber\\
 B_R(f)&=\sum_{i=1}^{n-1}\E_{U_n}|f(\sigma s_i)-f(\sigma)|,\qquad
 \cB_n(f)=B_L(f)+B_R(f).\label{eq:Bdef}
\end{align}
Thus $B_L$ exchanges neighboring labels, and $B_R$ exchanges neighboring positions.
\begin{theorem}[Two-sided permutation inequality]\label{thm:static}
For every $n\ge2$ and $f:\Sn\to\mathbb R$,
\begin{equation}\label{eq:static-main}
 \norm{f-\E_{U_n}f}_{L^1(U_n)}
 \le C\sqrt n\exp\!\left(C\sqrt{\log n\log\log(n+2)}\right)\cB_n(f).
\end{equation}
The optimal constant has order $n^{1/2+o(1)}$ in the sense that it is at least $c\sqrt n$ and at most $n^{1/2+\eta}$ times a constant depending on every fixed $\eta>0$.
\end{theorem}

Put $H_n(t)=H(\mu_t\mid U_n)$ and $\tau_n=\min\{t:H_n(t)\le1\}$.
\begin{theorem}[Intrinsic relative-entropy cutoff]\label{thm:entropy-main}
For every fixed $p\in(0,1)$,
\[
 \frac1{2\kappa}\le\liminf_n\frac{\tau_n}{n^2\log n}
 \le\limsup_n\frac{\tau_n}{n^2\log n}\le\frac1\kappa.
\]
Moreover,
\[
 \lim_{s\to\infty}\liminf_{n\to\infty}
 H_n(\tau_n-\lfloor sn^2\rfloor)=\infty,
 \qquad
 \lim_{s\to\infty}\limsup_{n\to\infty}
 H_n(\tau_n+\lceil sn^2\rceil)=0.
\]
For each fixed $s$, the earlier time is nonnegative for all sufficiently large $n$. These limits give an $O_p(n^2)$ entropy window about $\tau_n$. No sharp first-order location for $\tau_n$ is asserted.
\end{theorem}

\subsection{Proof architecture and scope of the inputs}
The lower bound starts with an exact first cosine mode of one card. Its lift to the deck has mean initially of order $n$ and standard deviation of order $\sqrt n$, uniformly in time. Its decay rate $\kappa/n^2$ identifies $t_n^\star$.

For the upper bound, couple two decks differing by one transposition. Their discrepancy evolves as two tagged positions until it is killed when adjacent. The resulting symmetric substochastic kernel $T_n$ has spectral radius $\rho_n$ with $n^2(1-\rho_n)\to\kappa$. Smoothing and an exact loss-of-mass identity yield
\[
 \sum_i\norm{\mu_{2\lceil n^2\rceil+s}-s_i\mu_{2\lceil n^2\rceil+s}}_{\TV}
 \le C_p\rho_n^s.
\]
Inversion symmetry makes the two-sided boundary of $d\mu_t/dU_n$ exactly four times this discrepancy sum. The static theorem then supplies the prefactor $\sqrt n\,n^{o(1)}$ needed for the upper cutoff.

The static proof is independent of the shuffle. A block table records coarse counts. Conditional on this table, a uniform permutation consists of two independent families of colored assignment words and independent permutations inside the cells. Assignment fluctuations control exceptional long gaps, permitting a low-congestion comparison from consecutive cell ranks to physical adjacent transpositions. If $C_m^\star$ denotes the optimal static constant and $A(x)=\max_{m\le\lfloor x\rfloor}C_m^\star$, the resulting recursion is
\[
 C_n^\star\le C\left(R^2\sqrt n+R\log n\,A(4n/R^2)\right).
\]
It closes at exponent $1/2+o(1)$. The auxiliary colored-word induction and every rare-set absorption are given quantitatively below.

The entropy proof is a separate branch: the same discrepancy coupling contracts the arbitrary-transposition metric, and the reverse kernel admits an almost-surely nonexpansive coupling. The Caputo--M\"unch--Salez theorem then contracts entropy on blocks of $O_p(n^2)$ steps. This branch uses neither the static theorem nor the TV upper bound.

The non-elementary inputs are explicitly stated: the complete-graph specialization of the interchange spectral-gap theorem, Salez's normalized multislice log-Sobolev inequality, and the general entropy-contraction theorem. We do not reprove those published theorems. Coarea, tensorization, concentration, the grid Nash estimate, and the analytic compactness step used here are supplied in the text. No numerical certificate, external code, or preceding conversational derivation is needed for any proof.

\section{Elementary tools and normalized external inputs}\label{sec:tools}
\subsection{Coarea, products, and paths}
For a finite law $\pi$, write $\Var_\pi f=\E_\pi(f-\E_\pi f)^2$ and
\[
 \Ent_\pi(F)=\E_\pi[F\log F]-(\E_\pi F)\log(\E_\pi F),\qquad F\ge0.
\]
This functional entropy is distinct from $H(\nu\mid\pi)$; they agree when $F=d\nu/d\pi$.
A boundary functional below is a sum of nonnegative weights times $|f(y)-f(x)|$.
\begin{lemma}[Coarea reduction]\label{lem:coarea}
If $\norm{\1_E-\pi(E)}_1\le C B(\1_E)$ for every set $E$, then $\norm{f-\E_\pi f}_1\le C B(f)$ for every real $f$. The best constant equals the supremum of $2\pi(E)(1-\pi(E))/B(\1_E)$ over nontrivial sets with nonzero boundary.
\end{lemma}
\begin{proof}
For $a\le\min f$, $b\ge\max f$,
$f-\E f=\int_a^b(\1_{\{f>u\}}-\pi(f>u))\,du$.
The triangle inequality bounds its $L^1$ norm by the integral of the centered-indicator norms. For every pair $x,y$,
$|f(x)-f(y)|=\int_a^b|\1_{\{f(x)>u\}}-\1_{\{f(y)>u\}}|\,du$.
Finite summation and integration commute, proving the upper bound. Indicators give the reverse comparison of optimal constants.
\end{proof}
Throughout indicator arguments, the empty and full sets have zero centered oscillation and are treated separately. When needed we complement a nontrivial set to arrange $0<\alpha\le1/2$.

\begin{lemma}[Product inequalities]\label{lem:products}
On a product of finitely many independent probability spaces, centered $L^1$ inequalities with constants $C_j$ tensorize with constant $\max_j C_j$, when their directional boundaries are added. Also
\begin{equation}\label{eq:var-tensor}
 \Var f\le\sum_j\E\Var_j f,
\end{equation}
where $\Var_j$ varies only the $j$th variable, holding the others fixed. For a function $h$ of independent Bernoulli($p$) bits,
\begin{equation}\label{eq:bernoulli-var}
 \Var h\le\sum_j p(1-p)\E|h(c^{j,1})-h(c^{j,0})|^2.
\end{equation}
\end{lemma}
\begin{proof}
For two independent variables $X,Y$, insert $\E_X f$ between $f$ and $\E f$ in the $L^1$ norm. Apply the $X$ inequality to the first difference and the $Y$ inequality to the second; Jensen bounds each $Y$ difference of $\E_X f$ by the expectation of the corresponding difference of $f$. Induction gives the assertion for any number of factors without multiplying constants.
For variance, $\Var f=\E_Y\Var_X f+\Var_Y\E_X f$, and the last term is at most $\E_X\Var_Y f$ by Jensen applied to the centered function in $Y$. Induct again. A single Bernoulli conditional variance is $p(1-p)$ times the squared two-value difference, giving \eqref{eq:bernoulli-var}.
\end{proof}

A \emph{multislice} $\Omega_{\boldsymbol k}$ consists of words $\omega\in\{1,\ldots,r\}^N$ containing $k_j$ copies of color $j$, where $k_j\ge1$ and $\sum_jk_j=N$. Its law is uniform. Zero classes, if present in a conditional composition, are discarded; a one-color space has one point. Let $\omega^{ab}$ exchange positions $a,b$. Define
\[
 B_{\rm RT}(f)=\sum_{a<b}\E|f(\omega^{ab})-f(\omega)|,
 \qquad B_{\rm path}(f)=\sum_{a=1}^{N-1}\E|f(\omega^{a,a+1})-f(\omega)|.
\]
For permutations these also define position-side boundaries. Left complete-transposition boundaries have the same sum, since conjugation permutes the set of all transpositions at each fixed state.

\begin{theorem}[Established complete-transposition inequalities]\label{thm:external-functional}
On a nontrivial multislice of length $N$, with smallest positive class $k$, put
\[
 \mathcal E_{\rm RT}(f)=\frac12\sum_{a<b}\E(f(\omega^{ab})-f(\omega))^2.
\]
Then
\begin{equation}\label{eq:external-poincare}
 \Var f\le \frac1N\mathcal E_{\rm RT}(f),
\end{equation}
and for a universal $C_{\rm LS}$,
\begin{equation}\label{eq:external-LSI}
 \Ent(f^2)\le\frac{C_{\rm LS}\log(N/k)}{N}
 \sum_{a<b}\E(f(\omega^{ab})-f(\omega))^2.
\end{equation}
For $r=N$, $k=1$, this is the uniform-permutation case. For a binary slice with both class sizes at least $N/3$, the logarithm in \eqref{eq:external-LSI} is bounded by $\log3$.
\end{theorem}
\begin{proof}[Source and normalization]
For the generator $\mathcal Lf=\sum_{a<b}(f(\omega^{ab})-f(\omega))$, each unordered transposition has rate one and its Dirichlet form is $\mathcal E_{\rm RT}$. On $\mathfrak S_N$, its spectral gap is $N$ by the complete-graph case of Caputo--Liggett--Richthammer \cite{CaputoLiggettRichthammer2010}. Functions on a multislice lift to functions on $\mathfrak S_N$ by identifying labels of the same color, preserving means, variances, and the displayed form, so \eqref{eq:external-poincare} follows. Salez's main multislice theorem \cite{Salez2021} gives a log-Sobolev constant at most $(4/\log2)\log(N/k)$ for form $(2N)^{-1}\sum_{a<b}\E(\Delta_{ab}f)^2$. Multiplying all rates by $N$ divides the functional-inequality constant by $N$. This proves \eqref{eq:external-LSI}, allowing a universal numerical factor in $C_{\rm LS}$. These published results are inputs, not new assertions proved here.
\end{proof}

\begin{lemma}[Routing and small sets]\label{lem:routing-basic}
On any multislice of length $N$,
\begin{equation}\label{eq:RT-route}
 B_{\rm RT}(f)\le2\sum_{i=1}^{N-1}i(N-i)\E|\Delta_{i,i+1}f|
 \le\frac{N^2}{2}B_{\rm path}(f).
\end{equation}
Consequently,
\begin{align}
 \norm{f-\E f}_1&\le C N B_{\rm path}(f),\label{eq:crude-path}\\
 \alpha\log(1/\alpha)&\le
 C N\log(N/k)B_{\rm path}(\1_E),\quad \alpha=\Pp(E),\label{eq:smallset-path}\\
 \alpha\log(1/\alpha)&\le
 \frac{C\log(N/k)}N B_{\rm RT}(\1_E).\label{eq:smallset-RT}
\end{align}
For a consecutive split into intervals of lengths $m,N-m$, averaging within the two intervals gives $\widetilde f$ with
\begin{equation}\label{eq:middle-averaging}
 \E|\widetilde f(\omega^{m,m+1})-\widetilde f(\omega)|
 \le2 B_{\rm path}(f).
\end{equation}
\end{lemma}
\begin{proof}
For $a<b$, perform swaps with indices $a,a+1,\ldots,b-1,b-2,\ldots,a$; their composition exchanges exactly $a,b$. Telescoping the difference of $f$ along this path and using invariance of the uniform law under every path prefix charges each edge in $[a,b)$ at most twice. There are $i(N-i)$ pairs straddling $i$, proving \eqref{eq:RT-route}.
For an indicator, $\mathcal E_{\rm RT}(\1_E)=B_{\rm RT}(\1_E)/2$, so \eqref{eq:external-poincare} gives $2\alpha(1-\alpha)\le B_{\rm RT}(\1_E)/N$. Coarea proves \eqref{eq:crude-path}. Equations \eqref{eq:smallset-path}--\eqref{eq:smallset-RT} are \eqref{eq:external-LSI} applied to an indicator, followed by routing when needed.
For \eqref{eq:middle-averaging}, Jensen and conjugation by the within-interval averaging permutations give the average of $\E|\Delta_{ab}f|$ over $a\le m<b$. Its routing coefficient at edge $i$ is at most
$2\#\{(a,b):a\le m<b,\ a\le i<b\}/[m(N-m)]\le2$.
\end{proof}

\subsection{Sampling without replacement}
\begin{lemma}[Hypergeometric tails]\label{lem:concentration}
If $X$ is the number of marked slots in a uniform sample of $s$ slots from $N$, of which $A$ are marked, and $\mu=sA/N$, then
\[
 \E e^{tX}\le(1-A/N+(A/N)e^t)^s\quad(t\in\mathbb R).
\]
In particular, for $0<\delta<1$,
\[
 \Pp(X\le(1-\delta)\mu)\le e^{-\delta^2\mu/2},
 \qquad
 \Pp(X\ge(1+\delta)\mu)\le e^{-\delta^2\mu/3}.
\]
For any fixed $b>1$, the upper tail at $b\mu$ is at most $e^{-c_b\mu}$. A uniform word of length $N$ with $k$ occurrences of a specified color avoids that color in a fixed interval of length $d$ with probability at most $e^{-kd/N}$.
\end{lemma}
\begin{proof}
Set $w_i=e^t$ on marked slots and $w_i=1$ otherwise. The moment generating function is $e_s(w)/\binom Ns$, where $e_s$ is the elementary symmetric polynomial. Among positive vectors with fixed sum, $e_s$ is maximized at the constant vector: replacing two entries $a,b$ by their average preserves their sum and increases their product, while
$e_s=e_s(\widehat a,\widehat b)+(a+b)e_{s-1}(\widehat a,\widehat b)+ab e_{s-2}(\widehat a,\widehat b)$.
Repeated averaging, or compactness and this inequality at a maximum, proves the claim, including $s=0,1$. Thus the generating function is bounded by the binomial one. Exponential Markov inequality, optimization in $t$, and the elementary estimates for $(1\pm\delta)\log(1\pm\delta)$ give the displayed tails. This is also the classical sampling-without-replacement comparison discussed in \cite{BardenetMaillard2015}.
The last probability is $\binom{N-k}{d}/\binom Nd$, interpreted as zero if $d>N-k$. Each factor in its product representation is at most $1-k/N$, proving the bound.
\end{proof}

\subsection{A square-root boundary estimate}\label{subsec:sqrt}
\begin{lemma}[Complete-transposition smoothing of indicators]\label{lem:sqrt-boundary}
Let $H$ be the product of symmetric groups on disjoint blocks whose sizes are between $m$ and $M$. Let $Q_H$ average over its action on a uniform finite orbit space, and let $R_{E,H}(x)$ be the number of within-block transpositions $t$ with $\1_E(tx)\ne\1_E(x)$. If all nontrivial block sizes are at least $m$, then
\begin{equation}\label{eq:sqrt-local}
 \norm{\1_E-Q_H\1_E}_1\le\frac{C}{\sqrt m}\E\sqrt{R_{E,H}}.
\end{equation}
In particular, for all transpositions on $\mathfrak S_n$,
\begin{equation}\label{eq:sqrt-boundary}
 2U_n(E)(1-U_n(E))\le\frac{C}{\sqrt n}\E_{U_n}\sqrt{R_E}.
\end{equation}
\end{lemma}
\begin{proof}
Let $\mathcal L=\sum_t(T_t-I)$ sum the allowed transpositions, where $T_tf(x)=f(tx)$ (the right-action convention works identically), and let $S_u=e^{u\mathcal L}$. Define
\[
 \Gamma(f,h)(x)=\frac12\sum_t(f(tx)-f(x))(h(tx)-h(x)),\qquad\Gamma(f)=\Gamma(f,f).
\]
Reversibility gives $-\E f\mathcal Lh=\E\Gamma(f,h)$. The sum of transpositions is central in the group algebra of $H$, so $T_tS_u=S_uT_t$. Jensen yields $\Gamma(S_uh)\le S_u\Gamma(h)$. Differentiation proves
\[
 S_u(h^2)-(S_uh)^2=2\int_0^uS_v\Gamma(S_{u-v}h)\,dv.
\]
Each integrand is at least $\Gamma(S_uh)$ by the preceding commutation inequality. Therefore $\Gamma(S_uh)\le1/(2u)$ when $|h|\le1$. By $L^1$ duality, reversibility, and Cauchy--Schwarz for the finite sum defining $\Gamma$,
\[
 \norm{f-S_uf}_1\le\sup_{|h|\le1}\int_0^u\E\sqrt{\Gamma(f)}\,
 \norm{\sqrt{\Gamma(S_vh)}}_\infty\,dv
 \le\sqrt{2u}\,\E\sqrt{\Gamma(f)}.
\]
The last bound is independent of the test function $h$.
On each $H$-orbit the generator has gap at least $m$: lift functions to the product of uniform symmetric groups, apply \eqref{eq:external-poincare} on each block, and tensorize variance. For an indicator of conditional mass $\alpha$, positivity preservation implies
\[
 \norm{f-S_uf}_{1,\mathrm{orbit}}
 =2\{\Var f-\Cov(f,S_uf)\}
 \ge(1-e^{-mu})\,2\alpha(1-\alpha).
\]
The covariance inequality follows by diagonalizing the symmetric generator. Average over orbits, take $u=1/m$, and use $\Gamma(\1_E)=R_{E,H}/2$ to prove \eqref{eq:sqrt-local}. A single block gives \eqref{eq:sqrt-boundary}.
\end{proof}

\section{Shuffle symmetries and the lower bound}\label{sec:lower}
\subsection{Ergodicity and the physical convention}
A cut vector is $c=(c_1,\ldots,c_{n-1})\in\{0,1\}^{n-1}$, with independent Bernoulli($p$) entries; $c_i=1$ means a cut after position $i$. Every packet reversal $R_c$ is an involution. Thus $P_n$ is symmetric and $U_n$ is reversible. Left multiplication by any fixed label permutation commutes with $P_n$. From the identity, inversion preserves the time-$t$ law, since
\[
 (R_{c_1}\cdots R_{c_t})^{-1}=R_{c_t}\cdots R_{c_1},
\]
and the cut vectors are independent with the same distribution. For $n\ge2$, the identity has one-step probability at least $p^{n-1}>0$. Each isolated adjacent transposition has probability $qp^{n-2}>0$. Adjacent transpositions generate $\Sn$, so $P_n$ is irreducible and aperiodic. The case $n=1$ is a stationary one-point space.

\begin{lemma}[Physical packet-order reversal]\label{lem:physical}
Let $J(i)=n+1-i$. For a fixed cut vector, reversing the order of packets while preserving each packet's internal order is the position permutation $R_cJ$. The physical and packet-reversal chains have identical total-variation and relative-entropy distances from uniform at every integer time.
\end{lemma}
\begin{proof}
Reading the one-line permutation $R_c$ backwards reverses packet order and reverses the internal reversal already made; this is exactly $R_cJ$. Reflection of cuts gives $JR_cJ=R_{c^{\rm ref}}$, with $c^{\rm ref}_i=c_{n-i}$ and the same cut law. At the kernel level, right reversal, denoted $\mathsf J$, therefore commutes with $P_n$. The physical kernel is $P_n\mathsf J$, so its $t$th power is $P_n^t\mathsf J^t$. A deterministic bijection preserves uniform measure, total variation, and relative entropy, proving the assertion. The identity is a kernel identity; individual packet reversals need not commute with $J$.
\end{proof}

\subsection{Exact one-card diagonalization}
Track a card currently in position $i$. Its packet contains $L$ positions to its left and $R$ to its right; its new position is $i+R-L$. With independent variables $G_L,G_R$ satisfying $\Pp(G=r)=pq^r$ for $r\ge0$,
\[
 L\stackrel d=\min(G_L,i-1),\qquad R\stackrel d=\min(G_R,n-i).
\]
Independence follows because the two explorations inspect disjoint cut bits. Let $K_n$ be the one-card kernel. Define, for $0\le k<n$,
\[
 \theta_k=\frac{k\pi}{n},\quad
 f_k(i)=\cos\bigl(\theta_k(i-\tfrac12)\bigr),\quad
 \ell_k=2(1-\cos\theta_k).
\]
\begin{proposition}[Finite-deck spectrum]\label{prop:onecard}
The $f_k$ form a basis and
\begin{equation}\label{eq:lambda-k}
 K_nf_k=\lambda_{n,k}f_k,\qquad
 \lambda_{n,k}=\frac{p^2+(-1)^kq^{n+1}\ell_k}{p^2+q\ell_k}.
\end{equation}
For $f=f_1$ and $\lambda_n=\lambda_{n,1}$,
\begin{align}
 \gamma_n:=1-\lambda_n
 &=\frac{q\ell_1(1+q^n)}{p^2+q\ell_1},\label{eq:gap-tag}\\
 -\log\lambda_n
 &=\frac\kappa{n^2}+O_p(n^{-4}+q^nn^{-2})\label{eq:tag-rate}
\end{align}
for all sufficiently large $n$, when $\lambda_n>0$.
\end{proposition}
\begin{proof}
For $|z|=1$ and $m\ge0$, summing a finite geometric series gives
\[
 A_m(z):=\E z^{\min(G,m)}
 =p\sum_{j=0}^{m-1}(qz)^j+(qz)^m
 =\frac{p+q^{m+1}(1-z)z^m}{1-qz}.
\]
Take $z=e^{i\theta_k}$, interpreting $z^a=e^{ia\theta_k}$ for every real exponent used here. Then $z^n=(-1)^k$ and
\[
 \E z^{X'-1/2}=z^{i-1/2}A_{n-i}(z)A_{i-1}(z^{-1}).
\]
The denominator is $(1-qz)(1-qz^{-1})=p^2+q\ell_k$. Its numerator is
\begin{align*}
 &p^2z^{i-1/2}
 +pq^i z^{1/2}(1-z^{-1})
 +pq^{n-i+1}z^{n-1/2}(1-z)\\
 &\hspace{35mm}+q^{n+1}\ell_kz^{n-i+1/2}.
\end{align*}
The middle terms have zero real part, because
$z^{1/2}(1-z^{-1})=2i\sin(\theta_k/2)$ and
$z^{n-1/2}(1-z)=-(-1)^k2i\sin(\theta_k/2)$.
The last term has real part $(-1)^kq^{n+1}\ell_k f_k(i)$. Taking real parts proves \eqref{eq:lambda-k}.
The sum formula
\[
 \sum_{j=1}^n\cos((j-\tfrac12)\theta)=\frac{\sin(n\theta)}{2\sin(\theta/2)}
\]
when its denominator is nonzero follows by summing $e^{i(j-1/2)\theta}$. Together with the product-to-sum identity it gives orthogonality of the $n$ cosine vectors. Their squared norms are $n$ for $k=0$ and $n/2$ otherwise, so they form a basis. Finally, $\ell_1=\pi^2/n^2+O(n^{-4})$, and expanding the rational expression and $-\log(1-x)$ proves \eqref{eq:gap-tag}--\eqref{eq:tag-rate}.
\end{proof}

\subsection{A full-deck eigenfunction and its fluctuations}
Write $X_\sigma(a)=\sigma^{-1}(a)$ for the position of label $a$. Define
\begin{equation}\label{eq:Phi}
 \Phi(\sigma)=\sum_{a=1}^nf(a)f(X_\sigma(a)).
\end{equation}
Linearity and \eqref{eq:lambda-k} give
\begin{equation}\label{eq:Phi-mean}
 \E[\Phi(\sigma_{t+1})\mid\sigma_t]=\lambda_n\Phi(\sigma_t),
 \qquad \E\Phi(\sigma_t)=\frac n2\lambda_n^t.
\end{equation}
Under $U_n$, every label position is uniform, and distinct positions are sampled without replacement. Therefore
\[
 \E_Uf(X(a))=0,\quad \E_Uf(X(a))^2=\frac12,\quad
 \E_U[f(X(a))f(X(b))]=-\frac1{2(n-1)}\ (a\ne b).
\]
Using $\sum_af(a)=0$ and $\sum_af(a)^2=n/2$ yields
\begin{equation}\label{eq:Phi-var-U}
 \E_U\Phi=0,\qquad \Var_U\Phi=\frac{n^2}{4(n-1)}.
\end{equation}

\begin{lemma}[Uniform conditional variance]\label{lem:onestep-var}
For every current deck,
\[
 \Var(\Phi(\sigma_1)\mid\sigma_0)\le B_p/n,
 \quad B_p=\frac{4\pi^2q(2-p)^2}{p^3}.
\]
Starting from the identity,
\begin{equation}\label{eq:Phi-var-time}
 \Var\Phi(\sigma_t)\le\frac{B_p}{n(1-\lambda_n^2)}=O_p(n)
\end{equation}
uniformly in time, for sufficiently large $n$.
\end{lemma}
\begin{proof}
Regard the next value as a function $H$ of the cut bits. Holding all other bits fixed, changing the bit at gap $j$ merges two neighboring packets of lengths $a,b$ or splits their union. Relative to reversing the packets separately, reversing their union moves the destination of every card in the left packet by $b$ and that of every card in the right packet by $a$ in the opposite direction. The interpolation $f(x)=\cos(\pi(x-1/2)/n)$ is $\pi/n$-Lipschitz, and the coefficient of every term in \eqref{eq:Phi} has modulus at most one. Hence
\[
 |H(c^{j,0})-H(c^{j,1})|\le2\pi ab/n.
\]
The lengths $a,b$ depend on disjoint collections of the other cut bits and are independent. Each is bounded stochastically by $1+G$, whose second moment is $(2-p)/p^2$ (differentiate the geometric series twice). Equation \eqref{eq:bernoulli-var} gives
\[
 \Var H\le\frac{4\pi^2pq(n-1)}{n^2}
 \left(\frac{2-p}{p^2}\right)^2\le\frac{B_p}{n}.
\]
The total-variance identity and \eqref{eq:Phi-mean} imply
$V_{t+1}\le B_p/n+\lambda_n^2V_t$, where $V_0=0$. Summing proves the assertion. For fixed $p$, \eqref{eq:gap-tag} makes its right side $O_p(n)$.
\end{proof}

\begin{proposition}[Lower side with a diffusive window]\label{prop:lower}
For every fixed $s>0$,
\[
 \liminf_n d_n(\lfloor t_n^\star-sn^2\rfloor)
 \ge1-C_pe^{-2\kappa s}.
\]
Consequently $t_{\rm mix}^{(n)}(\eps)\ge t_n^\star-O_{p,\eps}(n^2)$ for each fixed $\eps\in(0,1)$, and the lower side of \cref{thm:TV-main} holds.
\end{proposition}
\begin{proof}
Put $M_t=\E\Phi(\sigma_t)>0$ for large $n$. Testing $\{\Phi\ge M_t/2\}$ and applying Chebyshev under $\mu_t$ and $U_n$ gives
\[
 d_n(t)\ge1-\frac{4(\Var_{\mu_t}\Phi+\Var_U\Phi)}{M_t^2}.
\]
Equation \eqref{eq:tag-rate} is stronger than merely an $o(n^{-2})$ gap statement: multiplied by $t=O(n^2\log n)$ its error tends to zero. Thus at the time in the proposition,
$M_t=(1/2+o(1))\sqrt n\,e^{\kappa s}$.
Equations \eqref{eq:Phi-var-U} and \eqref{eq:Phi-var-time} finish the bound. Choose $s$ so the lower bound exceeds $\eps$ to obtain the mixing-time assertion. Markov kernels contract total variation, so $d_n$ is decreasing. For fixed $0<\delta<1$, $(1-\delta)t_n^\star\le t_n^\star-sn^2$ eventually for each fixed $s$; then let $s\to\infty$ to get the lower cutoff limit.
\end{proof}

\section{The killed two-card process}\label{sec:killed}
\subsection{An exact discrepancy coupling}
Suppose two decks differ by exchanging exactly two labels, currently at positions $i<j$. When the positions are not adjacent, use identical cut bits in the two decks. This preserves a single transposition discrepancy. When they occupy $i,i+1$, expose all bits except the separating bit, identically. Let $E_i$ be the event that the exterior boundaries isolate the interval $\{i,i+1\}$; a deck endpoint is an automatic boundary. Set
\[
 m_i=\1_{\{i>1\}}+\1_{\{i+1<n\}},\quad b=\min(p,q),\quad r_i=2bp^{m_i}.
\]
Then $\Pp(E_i)=p^{m_i}$. On $E_i$, use the following joint law for the separating bits:
\[
 \begin{array}{c|cc}
 &C'=0&C'=1\\\hline
 C=0&q-b&b\\
 C=1&b&p-b
 \end{array}
\]
Both marginal bits have probability $p$ of being one. They disagree with probability $2b$. Off $E_i$, use the same Bernoulli bit in both decks. Conditional on each exposed exterior configuration, every marginal bit remains Bernoulli($p$), so every marginal shuffle has independent cut bits with the required law. This is a maximum-disagreement coupling, not the usual maximal-agreement coupling of identical laws.
On disagreement, exactly one deck exchanges the two isolated cards, so the decks coalesce. Thereafter use identical updates.

Let
\[
 \cD=\{(i,j):1\le i<j\le n\}
\]
and let $Q_n$ be the kernel of the unordered positions of two distinct tagged cards under an ordinary shuffle. Every $R_c$ induces an involution on unordered pairs, so $Q_n$ is symmetric and stochastic. On $E_i$, the unordered pair remains $(i,i+1)$ regardless of the separating bit. Consequently the kernel before coalescence is exactly
\begin{equation}\label{eq:T-def}
 T_n=Q_n-\operatorname{diag}\bigl(r_i\1_{\{j=i+1\}}\bigr).
\end{equation}
Since $Q_n((i,i+1),(i,i+1))\ge p^{m_i}\ge r_i$, $T_n$ is entrywise nonnegative and symmetric, with row sums at most one. Such a matrix is called \emph{substochastic}; its missing row mass is killing.

\begin{lemma}[Coupling estimate]\label{lem:single-coupling}
If two decks $\sigma,\eta$ differ by a transposition at the pair of positions $z$, then
\[
 \norm{P_n^t(\sigma,\cdot)-P_n^t(\eta,\cdot)}_{\TV}
 \le(T_n^t\1)(z).
\]
Under this coupling, the arbitrary-transposition distance between the decks is either zero or one, and its expectation is $(T_n^t\1)(z)$.
\end{lemma}
\begin{proof}
The transition description shows inductively that $T_n^t\1$ is the survival probability of the discrepancy. For any coupling, differences of probabilities of events are at most the probability that its two coordinates disagree. Before coalescence these decks differ by exactly one transposition, proving both assertions.
\end{proof}

\subsection{Local comparison and spectral conventions}
All inner products and $\ell^a$ norms on $\cD$ in this section use \emph{counting measure}, not the uniform probability measure. Write $x\sim y$ if one coordinate of the pair changes by one and the resulting pair is still strictly ordered. Define the triangular nearest-neighbor energy with each undirected edge counted once:
\[
 E_{\rm nn}(g)=\sum_{\{x,y\}:x\sim y}(g(x)-g(y))^2.
\]
For a symmetric stochastic kernel $Q$, its Dirichlet form in counting measure is
$\ip g{(I-Q)g}=\frac12\sum_{x,y}Q(x,y)(g(x)-g(y))^2$.

\begin{lemma}[Uniform local bounds]\label{lem:pair-local-bounds}
Every nearest-neighbor pair edge, including one incident to an adjacent pair, satisfies
\[
 Q_n(x,y)\ge qp^4,\qquad Q_n(x,x)\ge p^4.
\]
Writing $k_p=2\min(p,q)p^2>0$, one has
\begin{equation}\label{eq:stir-dom}
 E_{\rm nn}(g)+\sum_i g(i,i+1)^2
 \le C_p\ip g{(I-T_n)g}.
\end{equation}
\end{lemma}
\begin{proof}
For a nearest-neighbor pair move, the moving tagged card and its destination form a two-position interval containing no other tagged card. Make that interval a packet of length two, and make the other tagged card a singleton. These requirements prescribe one absent cut and at most four present cuts. They are compatible even if the initial or final tagged positions are adjacent; shared requirements are the same present cut. Missing exterior boundaries only increase the event probability. Thus the move has probability at least $qp^4$. Making both tagged cards singletons prescribes at most four present cuts and proves the holding bound.
Symmetry gives $\ip g{(I-Q_n)g}\ge qp^4E_{\rm nn}(g)$. By \eqref{eq:T-def},
\begin{equation}\label{eq:killed-energy}
 \cE_n(g):=\ip g{(I-T_n)g}
 =\frac12\sum_{x,y}Q_n(x,y)(g(x)-g(y))^2
 +\sum_i r_i g(i,i+1)^2.
\end{equation}
Since $r_i\ge k_p$, \eqref{eq:stir-dom} follows. No nonadjacent endpoint square is attributed to the killing term.
\end{proof}

Let $\rho_n$ be the spectral radius of $T_n$, that is, the maximum modulus of its eigenvalues, and put $\beta_n=1-\rho_n$. The nearest-neighbor graph is connected and every adjacent state has positive killing. Thus $\rho_n<1$: a nonnegative eigenvector with eigenvalue one would attain a maximum, propagate that maximum along positive nearest-neighbor transitions, and contradict the positive killing row. Perron--Frobenius for a nonnegative matrix says that $\rho_n$ itself is an eigenvalue. Symmetry therefore gives
\[
 \norm{T_n}_{2\to2}=\rho_n,\qquad
 \beta_n=\inf_{g\ne0}\frac{\cE_n(g)}{\sum_{i<j}g(i,j)^2}.
\]
These facts do not require positive semidefiniteness. For all sufficiently large $n$, $0<\rho_n<1$; all logarithms of $\rho_n$ below are used in that range.

\subsection{The sharp principal decay rate}\label{subsec:pairlimit}
\begin{proposition}[Killed-pair spectral limit]\label{prop:killed-gap}
For fixed $p\in(0,1)$,
\begin{equation}\label{eq:killed-gap}
 n^2(1-\rho_n)\longrightarrow D\pi^2=\kappa.
\end{equation}
\end{proposition}
\begin{proof}
We give both variational bounds, specifying the interpolation and all limiting operations.

\paragraph{Interior jump law and uniform moments.}
In infinite volume a tagged displacement is $Z=G_R-G_L$, so direct convolution gives
\begin{equation}\label{eq:jump-law}
 k(u):=\Pp(Z=u)=\frac{p}{2-p}q^{|u|},\quad
 \sum_uuk(u)=0,\quad\sum_uu^2k(u)=2D.
\end{equation}
Finite-deck tagged displacements are bounded in absolute value by $G_L+G_R$, so they have uniformly exponentially decaying tails. Sorting the two destinations minimizes their sum of squared displacements from the ordered starting positions: if two destinations are reversed, switching their assignments lowers the squared cost by a nonnegative product of their separations. Thus the squared displacement of the ordered pair is at most the sum of squared tagged displacements. It follows that the ordered-pair displacement has bounded moments of every fixed order and a uniform exponential tail.

Fix $0<\delta<1/3$. For pairs whose endpoints and mutual separation are all at least $\delta n$, stop each packet exploration after $\lfloor\delta n/4\rfloor$ gaps. The four inspected sets are disjoint. The stopped explorations agree with independent infinite-volume explorations unless at least one survives the stopping length, an event of probability at most $C_pq^{\delta n/4}$. Hence for each fixed finite set of jumps $(u,v)$,
\begin{equation}\label{eq:pair-local}
 Q_n((i,j),(i+u,j+v))=k(u)k(v)+O_p(q^{\delta n/4})
\end{equation}
uniformly on these pairs. The ordered and tagged descriptions agree on the successful stopped event.

\paragraph{Rayleigh upper bound.}
Use $h_n(i,j)=f_1(i)-f_1(j)$. The cosine identities give
\[
 \sum_{i<j}h_n(i,j)^2=n\sum_if_1(i)^2-\Big(\sum_if_1(i)\Big)^2=n^2/2.
\]
On $\cT=\{(x,y):0<x<y<1\}$ put $h(x,y)=\cos(\pi x)-\cos(\pi y)$. Symmetry on the square gives
\[
 \int_\cT h^2=1/2,\qquad\int_\cT|\nabla h|^2=\pi^2/2.
\]
For fixed $M$, retain $|u|,|v|\le M$ in the interior. Taylor expansion of $h$ about the cell centers, \eqref{eq:pair-local}, and Riemann summation give the corresponding continuum quadratic form. Uniform second-moment tails bound the omitted jump energy by a quantity tending to zero as $M\to\infty$, uniformly in $n$. A boundary strip of width $\delta n$ contains $O(\delta n^2+n)$ states; the $C/n$-Lipschitz test and the uniform pair second moment bound its energy by $C_p(\delta+n^{-1})$. Let first $n\to\infty$, then $M\to\infty$, and then $\delta\downarrow0$. The kernel energy converges to $D\int_\cT|\nabla h|^2$. Also $|h_n(i,i+1)|\le\pi/n$, so its killing energy is $O_p(n^{-1})$. Thus
\begin{equation}\label{eq:pair-limsup}
 \limsup_n n^2\beta_n\le D\pi^2.
\end{equation}

\paragraph{Square extension and preservation of norm.}
Take Rayleigh minimizers $g_n$ normalized by $n^{-2}\sum_{i<j}g_n(i,j)^2=1$. Along any subsequence their energies are bounded by \eqref{eq:pair-limsup}. Define a square array
\[
 a_{ij}=\begin{cases}g_n(i,j),&i<j,\\0,&i=j,\\-g_n(j,i),&i>j.\end{cases}
\]
Let $E_\square$ be the sum of squared horizontal and vertical differences of this $n$ by $n$ array, each edge once. Edge counting gives the exact identity
\begin{equation}\label{eq:square-extension}
 E_\square(a)=2E_{\rm nn}(g_n)+4\sum_i g_n(i,i+1)^2.
\end{equation}
It is bounded by \eqref{eq:stir-dom}. Place $a_{ij}$ at $((i-1/2)/n,(j-1/2)/n)$, copy nearest values to the outer boundary, and interpolate affinely on rectangles divided by their southwest--northeast diagonal. This triangulation is invariant under exchanging coordinates; call the antisymmetric interpolant $\mathcal I_n$. Let $V_n$ equal $a_{ij}$ on $[(i-1)/n,i/n)\times[(j-1)/n,j/n)$.
Elementary affine interpolation gives
\begin{equation}\label{eq:interpolation}
 \int|\nabla \mathcal I_n|^2\le CE_\square(a),\quad
 \norm{\mathcal I_n-V_n}_2^2\le Cn^{-2}E_\square(a),\quad
 \norm{V_n}_2^2=2.
\end{equation}
To verify the bounds, subtract a corner value on each triangle. The affine variation is bounded by its vertex differences, and its gradient is bounded by those differences divided by the mesh. Multiply by triangle area and sum. Each side difference is charged a bounded number of times; a diagonal difference is bounded by twice the two squared orthogonal side differences. Copied boundary values add no new differences, and boundary half-rectangles have comparable aspect ratio. This proves both inequalities with a universal constant.

Here $H^1([0,1]^2)$ denotes square-integrable functions with square-integrable weak first derivatives: a weak derivative $\partial_i W$ satisfies $\int W\partial_i\varphi=-\int(\partial_i W)\varphi$ for every smooth test function $\varphi$ compactly supported in the open square. Strong $L^2$ convergence means convergence in its norm; weak $L^2$ convergence means convergence of inner products against every fixed $L^2$ function. The sequence $\mathcal I_n$ is bounded in this space by \eqref{eq:interpolation}. For completeness, its compactness in $L^2$ can be seen using the orthonormal square cosine basis $e_{ab}(x,y)=c_ac_b\cos(a\pi x)\cos(b\pi y)$, where $c_0=1$ and $c_j=\sqrt2$ for $j>0$. For any $W\in H^1$ put $\widehat W_{ab}=\int_{[0,1]^2}We_{ab}$. Parseval and the weak-derivative identity give
\[
 \norm{W}_2^2=\sum_{a,b\ge0}|\widehat W_{ab}|^2,\qquad
 \norm{\nabla W}_2^2=\pi^2\sum_{a,b\ge0}(a^2+b^2)|\widehat W_{ab}|^2.
\]
These identities follow first for smooth functions by even reflection to a periodic square and Fourier expansion, and then for $H^1$ functions by reflection and mollification. The squared mass above frequency $K$ is at most $C/K^2$. Extract convergent subsequences of finitely many low coefficients and use this tail bound. Thus a subsequence converges strongly in $L^2$ to an antisymmetric $U$ with norm squared two. Bounded gradients have a weakly convergent subsequence in the Hilbert space $L^2$; testing against smooth functions identifies them as the weak derivatives of $U$. Equation \eqref{eq:interpolation} gives the same strong limit for $V_n$. Restricting $U$ to $\cT$ gives a function $g$ with $\int_\cT g^2=1$.

\paragraph{Finite-displacement lower bounds.}
Define
\[
 \cT_\delta=\{(x,y):x>\delta,\ y<1-\delta,\ y-x>\delta\}.
\]
For a fixed vector $w=(u,v)\in\mathbb Z^2$, the scaled difference
$n[V_n(\cdot+w/n)-V_n]$ is bounded in $L^2$ on any such interior: telescoping along at most $|u|+|v|$ grid edges and Cauchy--Schwarz bound its squared norm by $C_wE_\square(a)$. For a smooth test $\varphi$ compactly supported in the interior,
\[
 \int n[V_n(x+w/n)-V_n(x)]\varphi(x)\,dx
 =\int V_n(x)n[\varphi(x-w/n)-\varphi(x)]\,dx.
\]
Strong $L^2$ convergence therefore identifies the weak limit as $w\cdot\nabla g$. Retain only jumps with $|u|,|v|\le M$ and starts whose cell centers lie in $\cT_{\delta/2}$. For large $n$ these include the cells needed for the integrals on $\cT_\delta$. Equation \eqref{eq:pair-local} applies uniformly. Its exponentially small coefficient error multiplies a bounded finite-displacement energy, so vanishes. Weak lower semicontinuity in $L^2$, applied to the finite list of differences, yields
\begin{equation}\label{eq:finite-jump-liminf}
 \liminf_n\cE_n(g_n)\ge
 D_M\int_{\cT_\delta}|\nabla g|^2,\quad
 D_M=\frac12\left(\sum_{|u|\le M}u^2k(u)\right)
                   \left(\sum_{|v|\le M}k(v)\right).
\end{equation}
Indeed symmetry cancels the mixed derivative term, and the two pure derivative coefficients agree. Killing and all omitted jump terms are nonnegative. Now let $M\to\infty$; \eqref{eq:jump-law} gives $D_M\to D$. Then let $\delta\downarrow0$, exhausting $\cT$, to obtain $\liminf_n\cE_n(g_n)\ge D\int_\cT|\nabla g|^2$.

\paragraph{The continuum constant.}
Antisymmetry gives $\int_{[0,1]^2}U=0$. In the same cosine expansion the constant coefficient vanishes, so
\[
 \int_{[0,1]^2}|\nabla U|^2\ge\pi^2\int_{[0,1]^2}U^2.
\]
Both integrals are twice their triangular counterparts. Hence $\int_\cT|\nabla g|^2\ge\pi^2$. Since the normalization makes $\cE_n(g_n)=n^2\beta_n$, this proves the lower bound matching \eqref{eq:pair-limsup}. Every subsequence attaining the original liminf admits this extraction, so the asserted limit holds. The diagonal Dirichlet condition is encoded by the antisymmetric extension; no unproved convergence of boundary traces is required.
\end{proof}

\subsection{A grid Nash inequality and discrete-time smoothing}
\begin{lemma}[Triangular grid Nash inequality]\label{lem:nash}
For counting-measure norms on $\cD$,
\begin{equation}\label{eq:nash}
 \norm{u}_2^4\le C\norm{u}_1^2
 \left(E_{\rm nn}(u)+n^{-2}\norm{u}_2^2\right).
\end{equation}
\end{lemma}
\begin{proof}
For $n\ge3$, extend symmetrically off the diagonal by $b_{ij}=u(\min(i,j),\max(i,j))$. Put $b_{ii}=u(i,i+1)$ for $i<n$ and $b_{nn}=u(n-1,n)$. Then $\norm b_2^2\ge2\norm u_2^2$, $\norm b_1\le4\norm u_1$, and $E_\square(b)\le CE_{\rm nn}(u)$. The only nontrivial comparison for the latter is between $u(i-1,i)$ and $u(i,i+1)$, routed by two edges through $u(i-1,i+1)$; every such path has bounded length and bounded congestion. Reflect $b$ evenly to a torus of side $2n$ by the index list $1,\ldots,n,n,\ldots,1$ in each coordinate. Torus norms and energies differ by fixed factors from the square ones.

On a torus of side $M=2n$, Fourier expansion gives for $1\le K\le cM$ a bound on squared $L^2$ mass at frequencies of magnitude at most $K$ of $CK^2M^{-2}\norm v_1^2$, since each Fourier coefficient has modulus at most $\norm v_1/M^2$. The high-frequency mass is at most $CM^2K^{-2}E_{\rm tor}(v)$, because the nearest-neighbor eigenvalue is bounded below by $cK^2/M^2$ there. If $M^2\norm v_2^2/\norm v_1^2$ is bounded, the constant-mode term alone gives the desired inequality. Otherwise choose an integer $K$ comparable to a sufficiently small multiple of $M\norm v_2/\norm v_1$, so that the low-frequency contribution is at most $\norm v_2^2/2$. Absorbing it into the high-frequency estimate gives $\norm v_2^4\le C\norm v_1^2E_{\rm tor}(v)$. Combining the cases and the extension comparisons proves \eqref{eq:nash}. The one-state triangle at $n=2$ is covered by enlarging $C$.
\end{proof}

\begin{lemma}[Pair-kernel smoothing]\label{lem:smoothing}
With $m=\lceil n^2\rceil$,
\[
 \max_{z\in\cD}\norm{Q_n^m(z,\cdot)}_2\le C_p/n,
 \qquad
 \max_{z\in\cD}\norm{T_n^m(z,\cdot)}_2\le C_p/n.
\]
\end{lemma}
\begin{proof}
Let $h=p^4$. The holding bound gives $Q_n=hI+(1-h)\widetilde Q_n$ with $\widetilde Q_n$ symmetric stochastic. Its eigenvalues lie in $[-1,1]$, so those of $Q_n$ are at least $2h-1$. For every such eigenvalue $a$, $1-a^2\ge2h(1-a)$; hence
\begin{equation}\label{eq:holding-comparison}
 I-Q_n^2\succeq2p^4(I-Q_n),
\end{equation}
where $\succeq$ denotes nonnegativity of the difference as a quadratic form. No holding probability of $1/2$ is assumed.
For the row distribution $v_s=Q_n^s(z,\cdot)$ put $u_s=\norm{v_s}_2^2$. Its $\ell^1$ norm is one. Equations \eqref{eq:holding-comparison}, \eqref{eq:nash}, and \cref{lem:pair-local-bounds} give
\[
 u_s-u_{s+1}=\ip{v_s}{(I-Q_n^2)v_s}
 \ge c_pu_s^2-C_pn^{-2}u_s.
\]
Above $K_p/n^2$ this is at least $c'_pu_s^2$ for a suitable $K_p$. While above that threshold,
$1/u_{s+1}-1/u_s\ge c'_p$; if $u_{s+1}=0$ the conclusion is immediate. Below the threshold, $\ell^2$ contraction keeps it there. Consequently $u_s\le C_p((s+1)^{-1}+n^{-2})$. Take $s=m$. Finally $0\le T_n^m\le Q_n^m$ entrywise, by induction on powers of nonnegative matrices, proving the second estimate.
\end{proof}

\subsection{Flux and adjacent discrepancies}
Write $a_t=T_n^t\1$ and $r=(I-T_n)\1$, so $r(i,i+1)=r_i\ge k_p$ and $r=0$ elsewhere.
\begin{lemma}[Loss-of-mass identity]\label{lem:flux}
For $m=\lceil n^2\rceil$ and every integer $s\ge0$,
\begin{align}
 \sum_{i=1}^{n-1}a_m(i,i+1)&\le C_p,\label{eq:flux-first}\\
 \max_z a_{m+s}(z)&\le C_p\rho_n^s,\label{eq:a-smooth}\\
 \sum_{i=1}^{n-1}a_{2m+s}(i,i+1)&\le C_p\rho_n^s.\label{eq:adj-surv}
\end{align}
\end{lemma}
\begin{proof}
Let $S_t=\ip\1{a_t}$. Symmetry yields
$S_t-S_{t+1}=\ip r{a_t}=\sum_i r_i a_t(i,i+1)=:b_t$.
Since $T_n\1\le\1$ and $T_n$ is nonnegative, $a_{t+1}\le a_t$ pointwise and $b_t$ decreases. Hence
\[
 (m+1)b_m\le\sum_{t=0}^mb_t=S_0-S_{m+1}\le|\cD|\le n^2/2.
\]
This gives \eqref{eq:flux-first}. By smoothing, symmetry, and the operator norm of $T_n$,
\[
 a_{m+s}(z)\le\norm{T_n^m(z,\cdot)}_2\norm{T_n^s\1}_2
 \le(C_p/n)\rho_n^s\sqrt{|\cD|}\le C_p\rho_n^s.
\]
Thus $a_{2m+s}=T_n^m a_{m+s}\le C_p\rho_n^s a_m$, and summation gives \eqref{eq:adj-surv}.
\end{proof}

\begin{corollary}[Aggregate adjacent discrepancy]\label{cor:adj-discrepancy}
With $\delta_i(t)=\norm{\mu_t-s_i\mu_t}_{\TV}$ and $m=\lceil n^2\rceil$,
\begin{equation}\label{eq:delta-sum}
 \sum_{i=1}^{n-1}\delta_i(2m+s)\le C_p\rho_n^s.
\end{equation}
Here $s_i\mu$ is the pushforward of $\mu$ by $\sigma\mapsto s_i\sigma$.
\end{corollary}
\begin{proof}
Relabeling commutes with the chain, so $s_i\mu_t$ is its law started from $s_i$. The discrepancy initially occupies $(i,i+1)$ in the identity deck. Apply \cref{lem:single-coupling} and sum \eqref{eq:adj-surv}.
\end{proof}

\section{One-dimensional and multivariate hypergeometric inequalities}\label{sec:hypergeom}
\subsection{A scalar log-concavity lemma}
A probability mass function $p_j$ on an integer interval is \emph{log-concave} if $p_j^2\ge p_{j-1}p_{j+1}$. Equivalently, its successive positive mass ratios decrease.
\begin{lemma}[Separation of conditional means]\label{lem:mean-gap}
If $Z$ has a finite integer log-concave law, then across every threshold $t$ with positive mass on both sides,
\[
 \E[Z\mid Z>t]-\E[Z\mid Z\le t]\ge\sqrt{\Var Z}.
\]
\end{lemma}
\begin{proof}
First let $W\ge0$ be integer log-concave and write $S_j=\Pp(W\ge j)$. On its positive support,
\[
 \frac{\Pp(W=j)}{S_j}
 =\left(1+\sum_{h\ge1}\prod_{\ell=0}^{h-1}
 \frac{p_{j+\ell+1}}{p_{j+\ell}}\right)^{-1}
\]
is increasing, because every successive product of mass ratios decreases when the starting index increases. Products stop at the upper end of the support, and terms beyond it are zero. Thus $S_{j+1}/S_j=1-p_j/S_j$ decreases; initial zeros in the mass function give survival ratios equal to one and cause no problem. Since $S_0=1$, $S_{i+j}\le S_iS_j$. If $m=\E W=\sum_{j\ge1}S_j$, then
\[
 \frac12\E[W(W-1)]=\sum_{i,j\ge1}S_{i+j}\le m^2,
 \qquad\Var W\le m(m+1).
\]
Let $\alpha=\Pp(Z\le t)$, $a=t-\E[Z\mid Z\le t]$, and $b=\E[Z\mid Z>t]-t-1$. Both are nonnegative. Reflection and translation of the two conditional laws preserve log-concavity, so their variances are at most $a(a+1)$ and $b(b+1)$. With $\Delta=a+b+1$, total variance is at most
\[
 \alpha a(a+1)+(1-\alpha)b(b+1)+\alpha(1-\alpha)\Delta^2.
\]
For fixed $a+b=\Delta-1$ this is convex in $a$ and therefore maximized at an endpoint. At $b=0$ it is $\alpha(2-\alpha)\Delta^2-\alpha\Delta\le\Delta^2$; the other endpoint follows by exchanging $\alpha$ and $1-\alpha$. This proves the result.
\end{proof}

For positive integers $A,B$ and $0\le s\le A+B$, consider
\[
 \pi_x=\frac{\binom Ax\binom B{s-x}}{\binom{A+B}s},\qquad
 \max(0,s-B)\le x\le\min(A,s).
\]
Use rates $u_x=(A-x)(s-x)$ and $d_x=x(B-s+x)$. A rate leading outside the support is zero.
\begin{lemma}[Scalar exchange inequality]\label{lem:scalar-exchange}
If the support is nontrivial and $v=\Var_\pi X$, then
\begin{equation}\label{eq:scalar-exchange}
 \norm{h-\E_\pi h}_1\le
 \frac1{(A+B)\sqrt v}\E_\pi\left[
 u_X|h(X+1)-h(X)|+d_X|h(X-1)-h(X)|\right].
\end{equation}
If $A,B\ge K$ and $(A+B)/6\le s\le5(A+B)/6$, its coefficient is at most $CK^{-3/2}$.
\end{lemma}
\begin{proof}
The ratio $\pi_{x+1}/\pi_x=(A-x)(s-x)/[(x+1)(B-s+x+1)]$ decreases, so the law is log-concave, and it also proves $\pi_xu_x=\pi_{x+1}d_{x+1}$. The mean is $\mu=As/(A+B)$, and $u_x-d_x=-(A+B)(x-\mu)$. Summing this stationary drift over $x\le t$ cancels interior flows and gives
\[
 \pi_tu_t=(A+B)\alpha_t(1-\alpha_t)(m_+(t)-m_-(t)),
\]
where $\alpha_t=\Pp(X\le t)$ and $m_\pm$ are the two conditional means. For independent copies $X,Y$, Jensen and telescoping give
\[
 \E|h(X)-\E h(X)|\le\E|h(X)-h(Y)|
 \le2\sum_t\alpha_t(1-\alpha_t)|h(t+1)-h(t)|.
\]
The directed boundary in \eqref{eq:scalar-exchange} is $2\sum_t\pi_tu_t|h(t+1)-h(t)|$. Apply \cref{lem:mean-gap} to obtain \eqref{eq:scalar-exchange}.
Sampling indicators without replacement gives
\[
 v=\frac{s(A+B-s)AB}{(A+B)^2(A+B-1)}.
\]
For the indicated central range of $s$, $v\ge cAB/(A+B)\ge cK$. Since $A+B\ge2K$, the last assertion follows.
\end{proof}

\subsection{Several count coordinates}
Let $K_1,\ldots,K_r\ge1$, $N=\sum_jK_j$, and $N/3\le m\le2N/3$. The count space and its law are
\[
 \mathcal X=\{x\in\mathbb Z^r:0\le x_j\le K_j,\ \sum_jx_j=m\},
 \qquad\pi(x)=\frac{\prod_j\binom{K_j}{x_j}}{\binom Nm}.
\]
Here $x$ counts selected slots of each type in a uniformly chosen $m$-subset of $N$ slots. Set $K=\min_jK_j$. Let $e_j$ denote the $j$th standard coordinate vector, and define the directed boundary
\[
 \cE_{\rm ex}(F)=\sum_{i\ne j}\E_\pi\left[
 X_i(K_j-X_j)|F(X-e_i+e_j)-F(X)|\right].
\]
Terms of rate zero are zero; $F$ need not be defined outside $\mathcal X$.
\begin{lemma}[Multivariate exchange inequality]\label{lem:hypergeom}
For universal constants $C,C_0$, if $K\ge C_0\log(r+1)$, then
\begin{equation}\label{eq:hypergeom}
 \norm{F-\E_\pi F}_1\le CK^{-3/2}\cE_{\rm ex}(F).
\end{equation}
The case $r=1$ is a one-point space and is understood separately.
\end{lemma}
\begin{proof}
We first prove a conditional-variance inequality for pair resampling, then convert it to an all-sets boundary estimate.

\paragraph{Pair heat-bath gap.}
For every unordered pair $\{i,j\}$, at rate one replace $(X_i,X_j)$ by its conditional law given all other coordinates. Write $Q_{ij}F=\E[F\mid X_\ell,\ell\notin\{i,j\}]$ for this conditional-averaging operator. The generator is $\mathcal L_{\rm hb}=\sum_{i<j}(Q_{ij}-I)$. It is reversible, and its Dirichlet form is
$\sum_{i<j}\E\Var(F\mid X_\ell,\ell\notin\{i,j\})$.
Use unit-transfer distance $d(x,y)=\frac12\sum_j|x_j-y_j|$, the shortest number of valid unit transfers connecting the states. Such a path exists by transferring units from coordinates where $x$ exceeds $y$ to those where it is smaller; capacities remain respected.

For neighboring states with discrepancy $e_i-e_j$, update the same coordinate pair. Updating $\{i,j\}$ has the same conditional total in the two copies and couples identically, eliminating the discrepancy. If exactly one discrepancy coordinate is updated, the two conditional totals differ by one. In a common uniform random ordering of the available $K_a+K_b$ slots, select the first $s$ versus the first $s+1$ slots. The conditional samples have the right laws and differ in one selected slot, so the resulting states remain at unit-transfer distance. Other updates preserve the discrepancy. Thus distance stays at most one and its survival probability is at most $e^{-t}$, since at every instant its current coordinate pair rings at rate one. Gluing neighboring couplings along a shortest path gives contraction by $e^{-t}$ for arbitrary starts.

For a nonconstant generator eigenfunction $F$ with eigenvalue $-\gamma$, its finite Lipschitz seminorm $\max_{x\ne y}|F(x)-F(y)|/d(x,y)$ is positive. The coupling implies $e^{-\gamma t}\operatorname{Lip}(F)\le e^{-t}\operatorname{Lip}(F)$, hence $\gamma\ge1$. Irreducibility follows because conditional pair resampling permits every valid unit transfer with positive probability. Diagonalizing the generator, which is self-adjoint in $L^2(\pi)$ by reversibility, therefore proves
\begin{equation}\label{eq:HB-poincare}
 \Var_\pi F\le\sum_{i<j}\E\Var(F\mid X_\ell,\ell\notin\{i,j\}).
\end{equation}

\paragraph{Conditional scalar bounds.}
For a pair $i,j$, put $A=K_i$, $B=K_j$, and $s=X_i+X_j$. Conditional on all other coordinates, $X_i$ has the scalar law above. The marginal $s$ is hypergeometric with sample size $A+B$ from $N$ slots of which $m$ are selected. Since $m/N\in[1/3,2/3]$, \cref{lem:concentration} applied also to unselected slots shows that
\[
 \Pp\{s\notin[(A+B)/6,5(A+B)/6]\}\le Ce^{-c(A+B)}\le Ce^{-cK}.
\]
For an indicator $F=\1_E$, apply \eqref{eq:HB-poincare}, multiply by two, and use \cref{lem:scalar-exchange} on good conditionings. On the others, $2\Var(F\mid\cdot)\le1/2$. Summation over pairs gives, with $\alpha=\pi(E)$,
\begin{equation}\label{eq:hyper-error}
 2\alpha(1-\alpha)\le CK^{-3/2}\cE_{\rm ex}(\1_E)+Cr^2e^{-cK}.
\end{equation}

\paragraph{All rare sets.}
Complement $E$ so $\alpha\le1/2$. Denote the last error by $\eta_0$. If $\alpha\ge2\eta_0$, it is at most half of $2\alpha(1-\alpha)$ and can be absorbed. Otherwise
$\log(1/\alpha)\ge cK-2\log r-O(1)\ge c'K$
when $C_0$ is sufficiently large; its size also handles the fixed constants for $r\ge2$. Lift $\1_E$ to the binary multislice of selected and unselected slots. Its smallest class is at least $N/3$. Its complete-transposition boundary is exactly $\cE_{\rm ex}(\1_E)$, since a swap of a selected slot of type $i$ with an unselected slot of type $j$ has $X_i(K_j-X_j)$ realizations; swaps with no count change contribute zero. Equation \eqref{eq:smallset-RT} yields
\[
 2\alpha(1-\alpha)\le2\alpha
 \le\frac C{NK}\cE_{\rm ex}(\1_E)
 \le CK^{-3/2}\cE_{\rm ex}(\1_E).
\]
In the last inequality $N\ge K\ge1$. Both cases prove the indicator estimate, and \cref{lem:coarea} proves \eqref{eq:hypergeom} for every $F$.
\end{proof}

\section{An all-functions colored-word inequality}\label{sec:words}
\begin{proposition}[Multislice path inequality]\label{prop:multislice-path}
There is a universal $C$ such that, for every nontrivial uniform multislice of length $N$ and minimum positive class $k$,
\begin{equation}\label{eq:multislice-path}
 \norm{F-\E F}_1\le C\frac{N^2}{k^{3/2}}B_{\rm path}(F).
\end{equation}
\end{proposition}
\begin{proof}
We prove the assertion for indicators by strong induction on $N$, with one uniform constant $K_\star$, then apply coarea. Put $A_{N,k}=N^2/k^{3/2}$. If $k\le N^{2/3}$, \eqref{eq:crude-path} implies the result with a universal constant, since $A_{N,k}\ge N$. Single-point factors have zero boundary and zero oscillation. Finitely many remaining small $N$ can be covered by increasing $K_\star$.

Suppose $k>N^{2/3}$ and $N$ is sufficiently large. Split the word at $m=\lfloor N/2\rfloor$, and let $X_j$ count color $j$ in the first part. Conditional on $X$, the two subwords are independent uniform multislices. By \cref{lem:concentration}, except on an event of probability
\begin{equation}\label{eq:word-bad}
 \eta\le2r e^{-ck},
\end{equation}
both subwords contain at least $0.45k_j$ copies of each color. For large $N$, each good factor's inductive scale is at most $\theta A_{N,k}$ with $\theta=0.9$, because
\[
 \frac{\lceil N/2\rceil^2}{(0.45k)^{3/2}}
 \bigg/\frac{N^2}{k^{3/2}}\longrightarrow
 \frac1{4(0.45)^{3/2}}<0.83.
\]

Let $f=\1_E$, $\alpha=\Pp(E)\le1/2$ after complementing, $v=2\alpha(1-\alpha)$, and $B=B_{\rm path}(f)$. For $h(X)=\E[f\mid X]$, the decomposition
$\norm{f-\E f}_1\le\E\norm{f-h(X)}_{1\mid X}+\norm{h-\E h}_1$
has conditional fluctuation at most $\theta K_\star A_{N,k}B$ on good compositions, by product tensorization and induction. On bad compositions it is at most $\eta/2$, because the centered $L^1$ norm of any indicator is at most $1/2$. The within-part boundaries are subsets of the full path boundary and preserve the conditioning $X$.

The law of $X$ is that in \cref{lem:hypergeom} with capacities $k_j$. Its hypotheses hold for all large $N$: $k>N^{2/3}$ and $r\le N/k$ imply $k\ge C_0\log(r+1)$. Regard $\widetilde f=h(X)$ as the within-part group average of $f$. The middle physical swap exchanges a first-part color $i$ with a second-part color $j$ with conditional probability
\[
 \frac{X_i(k_j-X_j)}{m(N-m)}.
\]
Therefore $\cE_{\rm ex}(h)=m(N-m)B_{\rm middle}(\widetilde f)$ exactly. Equations \eqref{eq:hypergeom} and \eqref{eq:middle-averaging} give
$\norm{h-\E h}_1\le C_1 A_{N,k}B$,
where $C_1$ is independent of $K_\star,N,k$. Altogether,
\begin{equation}\label{eq:word-induction-error}
 v\le(0.9K_\star+C_1)A_{N,k}B+\eta/2.
\end{equation}

If $\alpha\ge10\eta$, then $v\ge\alpha$ and $\eta/2\le v/20$. Absorption yields
\[
 v\le\frac{0.9K_\star+C_1}{0.95}A_{N,k}B\le K_\star A_{N,k}B
\]
when $K_\star\ge 40C_1$. This reserves a strict margin; an arbitrary factor-two absorption would not suffice.
If $\alpha<10\eta$, equation \eqref{eq:word-bad} and $k>N^{2/3}$ give $\log(1/\alpha)\ge c'k$ for all large $N$. Equation \eqref{eq:smallset-path} instead yields
\[
 v\le C\frac{N\log(N/k)}k B.
\]
Its coefficient divided by $A_{N,k}$ is
\[
 \frac{\sqrt{k}\log(N/k)}N
 =\frac{\sqrt{x}\log(1/x)}{\sqrt N}\le\frac2{e\sqrt N},\qquad x=k/N.
\]
The last inequality follows by differentiating $\sqrt{x}\log(1/x)$ on $(0,1]$. Thus one universal $K_\star$ handles these sets as well. Choose it once to exceed the preceding constants and the finite-base constants. The strong induction then closes without a level-dependent factor. Coarea gives the statement for every real $F$.
\end{proof}

\section{Recursive two-sided isoperimetry}\label{sec:static}
Define $C_1^\star=0$. For $n\ge2$, let $C_n^\star$ be the optimal constant in
$\norm{f-\E_{U_n}f}_1\le C_n^\star\cB_n(f)$, and for real $x\ge1$ set
\[
 A(x)=\max_{1\le m\le\lfloor x\rfloor}C_m^\star.
\]
The constants are finite by \eqref{eq:crude-path}; for example $C_n^\star\le Cn$. We will prove a recurrence with a strictly smaller argument of $A$.

\subsection{Block averaging and the coarse table}
Partition both $[n]$ label indices and position indices into the same $R$ consecutive intervals $I_1,\ldots,I_R$ of lengths $L_a$ differing by at most one. Let $H=\prod_a\mathfrak S_{I_a}$ and
\[
 \begin{aligned}
 M_{ab}(\sigma)&=\#\{i\in I_b:\sigma(i)\in I_a\},\qquad Q=Q_LQ_R,\\
 Q_Lf(\sigma)&=\E_{h\in H}f(h\sigma),\qquad
 Q_Rf(\sigma)=\E_{h\in H}f(\sigma h).
 \end{aligned}
\]
Both averages use the uniform law of $H$ and commute. For different label and position partitions, each average uses its own corresponding block subgroup. Two permutations have the same $M$ precisely when they are in the same orbit under these left and right actions: within each label block first match the labels assigned to each position block, and then reorder inside each position block. Averaging the finite transitive group action is uniform on its orbit. Therefore $Qf=\E[f\mid M]$ under $U_n$.
The table's law, obtained by counting assignments and then arrangements, is
\begin{equation}\label{eq:table-law}
 \pi(M)=\frac{\prod_aL_a!\prod_bL_b!}{n!\prod_{a,b}M_{ab}!},
\end{equation}
on nonnegative integer matrices with those row and column margins. This is the Fisher--Yates table law; its connection with permutation double cosets is discussed in \cite{Simper2024}.

\begin{lemma}[Stable block averaging]\label{lem:block-average}
For any consecutive label and position partitions (not necessarily the same),
\begin{equation}\label{eq:stable-average}
 B_L(Qf)\le4B_L(f),\qquad B_R(Qf)\le4B_R(f).
\end{equation}
If all block sizes are at most $L$, then also
\begin{equation}\label{eq:block-erasure}
 \norm{f-Qf}_1\le CL\cB_n(f).
\end{equation}
\end{lemma}
\begin{proof}
Left averaging commutes with right differences, so Jensen gives $B_R(Q_Lf)\le B_R(f)$. Every left difference internal to a label block vanishes. At a boundary between neighboring label blocks $I_a,I_{a+1}$, conjugation and Jensen bound its difference by the average of $\E|f((u\,v)\sigma)-f(\sigma)|$ over $u\in I_a,v\in I_{a+1}$. Routing each transposition along the intervening label edges charges any edge at most twice per neighboring-block pair. An original edge belongs to at most two neighboring-block unions; hence $B_L(Q_Lf)\le4B_L(f)$. Interchanging positions and labels gives $B_L(Q_Rf)\le B_L(f)$ and $B_R(Q_Rf)\le4B_R(f)$. Combining these separate directional estimates yields \eqref{eq:stable-average}, not a factor $16$.
For \eqref{eq:block-erasure}, average successively within the disjoint label blocks. On each orbit of a block of size $\ell$, lift to $\mathfrak S_\ell$ and apply \eqref{eq:crude-path}; other block averages commute with these internal differences and contract them. Summation costs at most $CL B_L(f)$. Position averaging costs at most $CL B_R(f)$, and the intervening averaging is an $L^1$ contraction. Add the two bounds.
\end{proof}

\begin{lemma}[Uniform coarse-table inequality]\label{lem:coarse-table}
If $n/R^2\ge C_0\log n$, for a sufficiently large universal $C_0$, every table-measurable real function satisfies
\begin{equation}\label{eq:coarse-table}
 \norm{F(M)-\E F(M)}_1\le CR^2\sqrt n\,\cB_n(F(M)).
\end{equation}
More generally, for any consecutive label and position partitions into $R$ blocks and smallest expected cell count $\lambda_*=\min_{a,b}L_aL_b/n\ge C_0\log n$, the complete-transposition bound is
\begin{equation}\label{eq:coarse-RT}
 \norm{F(M)-\E F(M)}_1\le
 \frac C{\lambda_*\sqrt n}B_{\rm RT}(F(M)).
\end{equation}
\end{lemma}
\begin{proof}
Consider a table-measurable event $E$, write $\alpha=U_n(E)\le1/2$ by complementing, and let $R_E$ be the number of complete transpositions crossing it. Each cell is hypergeometric with mean at least $\lambda_*$. On
$G_0=\{M_{ab}\ge\lambda_*/2\ \forall a,b\}$,
any crossing transposition must exchange two cards in different label blocks and different position blocks. If their cells are $(a,b),(c,d)$, all $M_{ab}M_{cd}$ exchanges between those two cells have the same table change and cross the event. Thus $R_E=0$ or $R_E\ge\lambda_*^2/4$, implying $\sqrt{R_E}\le2R_E/\lambda_*$ on $G_0$. By \cref{lem:concentration},
$\eta:=U_n(G_0^c)\le R^2e^{-\lambda_*/8}$.
On $G_0^c$ use $\sqrt{R_E}\le\sqrt{\binom n2}\le n$ in \eqref{eq:sqrt-boundary}. The resulting bound is
\begin{equation}\label{eq:coarse-error-correct}
 2\alpha(1-\alpha)
 \le\frac C{\lambda_*\sqrt n}B_{\rm RT}(\1_E)
 +C\sqrt n\,R^2e^{-\lambda_*/8}.
\end{equation}
In particular the exceptional error has a factor $\sqrt n$.
If $\alpha$ exceeds twice the last term, absorb that term using $2\alpha(1-\alpha)\ge\alpha$. Otherwise,
\[
 \log(1/\alpha)\ge\lambda_*/8-\tfrac12\log n-2\log R-O(1)
 \ge c\lambda_*,
\]
provided the universal threshold constant in $\lambda_*\ge C_0\log n$ is large enough; $R\le n$ is sufficient here. In this rare-set case \eqref{eq:smallset-RT} for permutations gives
\[
 2\alpha(1-\alpha)\le\frac{C\log n}{n\lambda_*}B_{\rm RT}(\1_E)
 \le\frac C{\lambda_*\sqrt n}B_{\rm RT}(\1_E),
\]
since $\log n\le C\sqrt n$. This proves \eqref{eq:coarse-RT} for all events and then for all functions by coarea. Routing \eqref{eq:RT-route} and $\lambda_*\ge cn/R^2$ for the near-equal blocks give \eqref{eq:coarse-table}. The hypotheses imply the needed comparability once finitely many small cases are absorbed into constants.
\end{proof}

\subsection{Exact product coordinates on a fixed-table fiber}
Fix a feasible table $M$. For each label block $I_a$, list its labels increasingly and record the position-block index of each; this is a word $A_a$ with $M_{ab}$ copies of color $b$. For each position block $I_b$, record in increasing position order the label-block index of the card there; this is a word $B_b$ with $M_{ab}$ copies of color $a$.
For each cell $(a,b)$, let $\ell_1<\cdots<\ell_m$ be its assigned labels and $x_1<\cdots<x_m$ its assigned positions, where $m=M_{ab}$. Define
\[
 \Pi_{ab}(j)=k\quad\text{when}\quad\sigma(x_j)=\ell_k.
\]
Thus each cell permutation uses the same label-at-position convention as $\sigma$. Empty cells and one-point cells have their unique trivial permutation.

\begin{lemma}[Product coordinates]\label{lem:fiber-product}
Under $U_n(\cdot\mid M)$, the words $A_a,B_b$ and the permutations $\Pi_{ab}$ are all independent and uniform on their respective multislices and symmetric groups.
\end{lemma}
\begin{proof}
Any choice of the two word families specifies the increasing label and position lists in each cell. Any choice of the cell permutations then reconstructs a unique deck by $\sigma(x_j)=\ell_{\Pi_{ab}(j)}$. Conversely every deck yields these choices. The total number is
\[
 \left(\prod_a\frac{L_a!}{\prod_bM_{ab}!}\right)
 \left(\prod_b\frac{L_b!}{\prod_aM_{ab}!}\right)
 \prod_{a,b}M_{ab}!
 =\frac{\prod_aL_a!\prod_bL_b!}{\prod_{a,b}M_{ab}!},
\]
agreeing with the fiber count in \eqref{eq:table-law}. Uniformity on a finite Cartesian product is the product of the uniform laws, proving independence.
\end{proof}

Write $\omega=(A_1,\ldots,A_R,B_1,\ldots,B_R)$ and $\Pi=(\Pi_{ab})$. An adjacent label exchange within $I_a$ between two different position-block colors swaps those colors in $A_a$ and changes no other word or cell permutation. Indeed replacing a label by the neighboring integer leaves its rank among the labels of its own cell unchanged. If the two colors agree, the exchange instead left-multiplies the corresponding $\Pi_{ab}$ by its adjacent rank transposition. Position exchanges have the analogous effect on $B_b$ or right-multiply a cell permutation. These statements specify the action on the product coordinates, not merely the resulting cell counts.

Define the physical within-fiber boundary
\begin{align*}
 B_M(f)={}&\E_{U_n(\cdot\mid M)}
 \sum_{i:\,i,i+1\in I_a\text{ for some }a}|f(s_i\sigma)-f(\sigma)|\\
 &+\E_{U_n(\cdot\mid M)}
 \sum_{i:\,i,i+1\in I_b\text{ for some }b}|f(\sigma s_i)-f(\sigma)|.
\end{align*}
Both move types preserve $M$. Let $B_{\rm assign}(f)$ be the expected sum of differences made by swapping neighboring colors in one of the words, holding $\Pi$ and the other words fixed. Equal-color swaps contribute zero. The action description gives $B_{\rm assign}(f)\le B_M(f)$.
For fixed assignments, define
\begin{equation}\label{eq:internal-boundary-def}
 b_{\rm int}(f;\omega)=\sum_{a,b}\sum_{j=1}^{M_{ab}-1}\E_\Pi
 \left[|f(\omega,s_j^{ab}\Pi)-f(\omega,\Pi)|
 +|f(\omega,\Pi s_j^{ab})-f(\omega,\Pi)|\right],
\end{equation}
where $s_j^{ab}$ acts only in cell $(a,b)$ by the adjacent transposition of ranks $j,j+1$. The expectation includes all cell permutations, not the assignments.

\subsection{Balanced fibers, gaps, and exceptional assignments}
Call a table balanced if
\begin{equation}\label{eq:balanced}
 n/(3R^2)\le M_{ab}\le3n/R^2\quad\text{for all }a,b.
\end{equation}
On such a fiber, every assignment word has length comparable to $n/R$ and minimum color count comparable to $n/R^2$. Equation \eqref{eq:multislice-path} and product tensorization imply, for each fixed $\Pi$ and then averaged in $\Pi$,
\begin{equation}\label{eq:assignment-control}
 \E_\Pi\norm{f-\E_\omega f}_{1,\omega}
 \le CR\sqrt n\,B_{\rm assign}(f).
\end{equation}

Let $d_R=\lceil C_gR\log n\rceil$ for a sufficiently large universal $C_g$. Let $G$ be the event that every gap between consecutive occurrences of a color in every assignment word is at most $d_R$. This event depends only on $\omega$. Each color has density at least $c/R$. If a gap exceeds $d_R$, some interval of $d_R$ consecutive coordinates avoids that color. \Cref{lem:concentration} and a union bound over $2R$ words, $R$ colors, and at most $n$ starting points per word give
\begin{equation}\label{eq:gap-bad}
 \eta:=\Pp_\omega(G^c)\le2R^2n\,e^{-cd_R/R}\le1/100
\end{equation}
for all sufficiently large $n$ under $R\le n$; increasing $C_g$ fixes the constant. If $d_R$ exceeds a word's length, that word cannot have such a gap and contributes no bad intervals.

\begin{lemma}[Exceptional assignments]\label{lem:exceptional}
For an indicator $f$ on a balanced fiber, complement it within the fiber so that $\alpha=\E[f\mid M]\le1/2$. Then
\begin{equation}\label{eq:exceptional-var}
 (1-2\eta)\,2\Var(f\mid M)
 \le2\E_\omega[\1_G\Var_\Pi f]
 +6\E_\Pi\Var_\omega f.
\end{equation}
\end{lemma}
\begin{proof}
The product variance inequality gives $\Var f\le\E_\omega\Var_\Pi f+V_A$, where $V_A=\E_\Pi\Var_\omega f$. For an indicator, $\Var_\Pi f\le\E_\Pi f$. For fixed $\Pi$, set $h(\Pi)=\E_\omega f$. Since the law of $\omega$ is independent of $\Pi$,
\[
 \E_\omega[f\1_{G^c}]\le\eta h(\Pi)+\E_\omega|f-h(\Pi)|
 =\eta h(\Pi)+2\Var_\omega f.
\]
Averaging gives $\E[\1_{G^c}\Var_\Pi f]\le\eta\alpha+2V_A$, so
$\Var f\le\E[\1_G\Var_\Pi f]+\eta\alpha+3V_A$.
Finally $\alpha\le2\Var f$ because $\alpha\le1/2$. Rearrange and multiply by two. No independence of $f$ from $G$ is assumed.
\end{proof}

\begin{lemma}[Routing consecutive cell ranks]\label{lem:cell-routing}
On a balanced fiber,
\begin{equation}\label{eq:cell-routing}
 \E_\omega[\1_G\,b_{\rm int}(f;\omega)]\le C d_R B_M(f).
\end{equation}
\end{lemma}
\begin{proof}
Consider the label side; the position side is identical with the actions reversed. A virtual adjacent-rank move inside a cell exchanges labels $a<b$ that are consecutive occurrences of a fixed position-block color in one word $A_u$. On $G$, $b-a\le d_R$. Perform in chronological order the physical label swaps with indices
\begin{equation}\label{eq:cell-word}
 a,a+1,\ldots,b-1,b-2,\ldots,a.
\end{equation}
Their composition exchanges $a,b$ and fixes every other label. All lie in the parent label block, so every intermediate deck belongs to the same table fiber, even though its assignments need not remain in $G$.

We count paths using a fixed intermediate directed label edge $(\tau,s_i\tau)$. On the outward leg of \eqref{eq:cell-word}, the moving color is the color currently at label index $i$ in $\tau$'s assignment word. The next occurrence of that color strictly to the right determines $b$, because the original endpoints were consecutive occurrences. There are at most $d_R$ choices of the original $a$. Given $a,b,i$ and the leg, the source deck is uniquely reconstructed by reversing the known path prefix. On the return leg the moving color is at index $i+1$; its next occurrence strictly beyond $i+1$ determines $b$, and again there are at most $d_R$ possible $a$. Thus each directed edge is used by at most $2d_R$ source paths. The reconstructed endpoints also determine the cell; there is no further factor for the number of cells. For $b=a+1$ the path has one outward step and no return step, consistent with the bound.

Telescoping $|f(\text{end})-f(\text{source})|$ along each path uses no squared path-length factor. The fixed-table measure is uniform, so source and intermediate decks have the same weight. Summing over all sources with $\omega\in G$, all virtual label moves, and then all physical directed edges proves the label-side comparison with factor $2d_R$. The position-side proof contributes to the corresponding position part of $B_M$, not an extra number-of-cells factor. This proves \eqref{eq:cell-routing}.
\end{proof}

\begin{proposition}[Balanced-fiber inequality]\label{prop:fiber-recursion}
Every balanced fiber satisfies, for all real $f$,
\begin{equation}\label{eq:fiber-recursion}
 \norm{f-\E[f\mid M]}_{1\mid M}
 \le C\left[R\sqrt n+R\log n\,A(4n/R^2)\right]B_M(f).
\end{equation}
\end{proposition}
\begin{proof}
Take an indicator and use \cref{lem:exceptional}, complementing it within this fiber if necessary; all boundaries and variances are unchanged by complementation. Equation \eqref{eq:assignment-control} bounds $2\E_\Pi\Var_\omega f$ by $CR\sqrt n B_M(f)$. At each fixed assignment, cell permutations are independent and have sizes at most $3n/R^2$. Product tensorization of the definition of $A$ gives
\[
 2\Var_\Pi f\le A(4n/R^2)b_{\rm int}(f;\omega).
\]
Multiply by $\1_G$, average, and use \cref{lem:cell-routing}. Since $\eta\le1/100$, the fixed factor $(1-2\eta)^{-1}$ is bounded. This proves the indicator inequality with a universal constant. Coarea on the fixed fiber proves \eqref{eq:fiber-recursion}.
\end{proof}

\subsection{The error-free recurrence}
\begin{proposition}[Static recurrence]\label{prop:static-recursion}
For universal constants $C,C_0$ and all sufficiently large $n$, if the integer $R$ satisfies $4\le R\le n^{1/8}$ and $n/R^2\ge C_0\log n$, then
\begin{equation}\label{eq:static-recursion}
 C_n^\star\le C\left[R^2\sqrt n+R\log n\,A(4n/R^2)\right].
\end{equation}
\end{proposition}
\begin{proof}
For the near-equal blocks, the cell means $L_aL_b/n$ equal $(n/R^2)(1+O(R/n))$. The fixed thresholds in \eqref{eq:balanced} therefore have, by \cref{lem:concentration}, an exceptional probability
\[
 \eta_T:=U_n(M\text{ unbalanced})\le CR^2e^{-cn/R^2}.
\]
Let $f=\1_E$, $\alpha=U_n(E)\le1/2$, and $v=2\alpha(1-\alpha)$. Split
\[
 \norm{f-\E f}_1\le\E\norm{f-Qf}_{1\mid M}+\norm{Qf-\E f}_1.
\]
On balanced tables use \eqref{eq:fiber-recursion}; on the others the centered indicator norm is at most $1/2$, contributing at most $\eta_T/2$. Since the within-fiber moves are physical moves, $\E B_M(f)\le\cB_n(f)$. Equations \eqref{eq:coarse-table} and \eqref{eq:stable-average} handle the second term. Thus, with $\mathcal C_n(R)$ the right coefficient in \eqref{eq:static-recursion},
\[
 v\le\mathcal C_n(R)\cB_n(f)+\eta_T/2.
\]
If $\alpha\ge\eta_T$, absorb the error using $v\ge\alpha$. Otherwise
$\log(1/\alpha)\ge cn/R^2-2\log R-O(1)\ge c'n/R^2$
by the stated threshold. Apply \eqref{eq:smallset-path} in the permutation case $k=1$ to obtain
\[
 v\le CR^2\log n\,\cB_n(f)\le CR^2\sqrt n\,\cB_n(f).
\]
Both cases give \eqref{eq:static-recursion} for indicators after enlarging its universal $C$. Coarea proves the recurrence for the optimal all-functions constant. Finally $\lfloor4n/R^2\rfloor\le n/4<n$, so the recurrence is genuinely downward.
\end{proof}

\subsection{Closing the recursion at the square-root exponent}
\begin{proof}[Proof of \cref{thm:static}]
First fix $\eta>0$. Choose $a>1/(2\eta)$ and $R=\lceil(\log n)^a\rceil$. For large $n$ the recurrence is admissible. If $C_m^\star\le K_\eta m^{1/2+\eta}$ for all $m<n$, then the recursive term in \eqref{eq:static-recursion} is at most
\[
 C_\eta K_\eta n^{1/2+\eta}\frac{\log n}{R^{2\eta}},
\]
which is at most half the target for all large $n$. The first term is $n^{1/2+o(1)}$ and occupies the other half for large $n$. Increase $K_\eta$ to handle the finitely many smaller indices using the crude finite constants. Strong induction proves the exponent $1/2+\eta$ for every $\eta>0$.

For the stated explicit envelope put $s=\log n$, $F(s)=\sqrt{s\log(s+2)}$, and choose $R=\lfloor e^{F(s)}\rfloor$. For large $n$, $4\le R\le n^{1/8}$ because $F(s)=o(s)$, and $n/R^2\ge C_0\log n$. Write $n'=\lfloor4n/R^2\rfloor$, so $s'=\log n'=s-2F(s)+O(1)$. The derivative
\[
 F'(u)=\frac{\log(u+2)+u/(u+2)}{2\sqrt{u\log(u+2)}}
\]
is bounded below by $c\sqrt{\log s/s}$ uniformly on $[s',s]$ for large $s$, since $s'=s-o(s)$. Therefore
\begin{equation}\label{eq:envelope-drop}
 F(s)-F(s')\ge c\log(s+2).
\end{equation}
Assume inductively $C_m^\star\le K\sqrt m\,e^{K_1F(\log m)}$ for $m<n$; the envelope is increasing, so it also bounds $A(n')$. Relative to $K\sqrt n e^{K_1F(s)}$, the first term of \eqref{eq:static-recursion} is at most $(C/K)e^{-(K_1-2)F(s)}$. The recursive term is at most
\[
 C R\sqrt{n'/n}\,\log n\,e^{-K_1(F(s)-F(s'))}
 \le 2C s\,(s+2)^{-cK_1},
\]
using $R\sqrt{n'/n}\le2$. Choose $K_1>2$ and $cK_1>2$, so both ratios are at most $1/2$ beyond a fixed size; then increase $K$ to cover the finite remainder. For $n\ge2$, the two positive quantities $\log(\log n+2)$ and $\log\log(n+2)$ are comparable up to universal factors: their ratio tends to one and is bounded on the remaining range. Changing constants therefore gives \eqref{eq:static-main}. The lower $c\sqrt n$ bound is proved independently in \cref{prop:static-lower} below.
\end{proof}

\section{Total-variation cutoff and its precise scope}\label{sec:tvproof}
Let $g_t=d\mu_t/dU_n$. The equality in law of a random time-$t$ deck and its inverse implies equality of its left and right adjacent gradient sums. Moreover,
\[
 \E_{U_n}|g_t(s_i\sigma)-g_t(\sigma)|
 =2\norm{s_i\mu_t-\mu_t}_{\TV}=2\delta_i(t).
\]
To check the right side, invert $\sigma$ in the right-gradient expectation and use $g_t(\sigma^{-1})=g_t(\sigma)$. Thus
\begin{equation}\label{eq:B-density}
 \cB_n(g_t)=4\sum_{i=1}^{n-1}\delta_i(t).
\end{equation}
Combining \cref{thm:static,cor:adj-discrepancy}, with $m=\lceil n^2\rceil$, gives
\begin{equation}\label{eq:tv-upper-master}
 d_n(2m+s)\le C_p\sqrt n
 \exp\!\left(C\sqrt{\log n\log\log(n+2)}\right)\rho_n^s,
 \qquad s\in\mathbb Z_{\ge0}.
\end{equation}
For any fixed $\delta>0$, take $2m+s=\lceil(1+\delta)t_n^\star\rceil$, which has $s\ge0$ for large $n$. By \cref{prop:killed-gap},
\[
 -\log\rho_n=\frac{\kappa+o(1)}{n^2}.
\]
The logarithm of the right side of \eqref{eq:tv-upper-master} is
$-(\delta/2)\log n+o(\log n)$ and tends to $-\infty$. The lower limit is \cref{prop:lower}. This proves \cref{thm:TV-main}, including the mixing-time asymptotic for every fixed $\eps$ by monotonicity of total variation.

\subsection{What the upper estimate does and does not say about a window}
For all sufficiently large $n$, a direct consequence of \eqref{eq:tv-upper-master} is
\begin{equation}\label{eq:tv-quantile}
 t_{\rm mix}^{(n)}(\eps)\le2\lceil n^2\rceil+
 \left\lceil\frac{\tfrac12\log n+C\sqrt{\log n\log\log(n+2)}+\log(C_p/\eps)}{-\log\rho_n}\right\rceil,
\end{equation}
where $C_p$ can be enlarged to at least one. Its spectral center is
\[
 t_n^{\rm sp}=\frac{\log n}{2(-\log\rho_n)}.
\]
The prefactor in the static inequality contributes an upper-time correction of order $n^2\sqrt{\log n\log\log n}$ relative to this exact spectral center. It does \emph{not} prove a window of that size about $t_n^\star$. To see the distinction, write
\[
 -\log\rho_n=\frac\kappa{n^2}(1+\varepsilon_n),\qquad\varepsilon_n\to0.
\]
Then
\[
 t_n^{\rm sp}-t_n^\star
 =-\frac{n^2\log n}{2\kappa}\frac{\varepsilon_n}{1+\varepsilon_n}.
\]
No rate for $\varepsilon_n$ was proved in \cref{prop:killed-gap}. It may not be absorbed into the displayed prefactor correction without another estimate. Thus the result here is leading-order cutoff, together with the $O_p(n^2)$ lower-side window, not a sharp TV window or a limiting TV profile. Improving the static prefactor alone would not remove this spectral-center issue.

\section{Relative-entropy cutoff}\label{sec:entropy}
\subsection{The external contraction theorem and its hypotheses}
For a finite metric space $(\mathcal X,d)$, define
\[
 W_{1,d}(\nu,\nu')=\min_{(X,Y)}\E d(X,Y),
\]
where the minimum is over couplings with the indicated marginals. The minimum exists because the coupling set is a compact finite-dimensional polytope. Its triangle inequality follows by gluing two couplings along their common marginal and applying the triangle inequality for $d$.
For a finite irreducible Markov kernel $K$ with invariant full-support law $\pi$, its reversed kernel is
\[
 K^*(x,y)=\frac{\pi(y)K(y,x)}{\pi(x)}.
\]
\begin{theorem}[Caputo--M\"unch--Salez]\label{thm:CMS}
Suppose for every $x,y$ there is a coupling of $K^*(x,\cdot)$ and $K^*(y,\cdot)$ whose distance is almost surely at most $d(x,y)$. If, for $0\le a\le1$,
\[
 W_{1,d}(\nu K,\nu'K)\le aW_{1,d}(\nu,\nu')
 \quad\text{for all initial laws }\nu,\nu',
\]
then
\[
 H(\nu K\mid\pi)\le a H(\nu\mid\pi)\quad\text{for every }\nu.
\]
\end{theorem}
This is Theorem 1 with Assumption 1 of \cite{CaputoMunchSalez2025}. The almost-sure reverse-kernel condition is called nonnegative sectional curvature there. We use this precisely stated published theorem, not an assertion that every Wasserstein contraction transfers to entropy.

Let $d_{\rm tr}(\sigma,\eta)$ be the minimum number of arbitrary label transpositions transforming $\sigma$ into $\eta$. It is a metric: reversing a sequence gives symmetry and concatenating sequences gives the triangle inequality. Its generating set is invariant under conjugation, so multiplication of both arguments on either side by the same permutation preserves this distance; it is \emph{bi-invariant}.
Define
\[
 a_n(t)=\max_{z\in\cD}(T_n^t\1)(z).
\]
For neighbors in this metric, \cref{lem:single-coupling} gives expected distance at most $a_n(t)$. Apply the triangle inequality for $W_1$ along a shortest transposition path to obtain the same contraction for arbitrary point-mass starts. Conditional mixing of these couplings over an optimal coupling of initial laws gives
\begin{equation}\label{eq:W-contract}
 W_{1,d_{\rm tr}}(\nu P_n^t,\nu'P_n^t)
 \le a_n(t)W_{1,d_{\rm tr}}(\nu,\nu').
\end{equation}
This does not require a single coupling to realize all starting pairs simultaneously.

Apply \cref{thm:CMS} at the scale $K=P_n^t$ for an integer $t\ge1$. The kernel is irreducible: each positive one-step move also occurs in $t$ steps by inserting $t-1$ identity updates. It is symmetric, so $K^*=K$ with invariant law $U_n$. For the reverse-kernel assumption, multiply both starting decks on the right by the same random $t$-step product. This is a coupling of the required reverse laws, and its distance is exactly $d_{\rm tr}(\sigma,\eta)$ by bi-invariance. Thus the hypothesis holds at every scale used in \eqref{eq:W-contract}. We conclude
\begin{equation}\label{eq:entropy-contract}
 H(\nu P_n^t\mid U_n)\le a_n(t)H(\nu\mid U_n)
\end{equation}
for every initial law. The formula also holds at $t=0$ with $a_n(0)=1$.

\subsection{Entropy location and a diffusive window}
By \eqref{eq:a-smooth}, $a_n(m+s)\le C_p\rho_n^s$ for $m=\lceil n^2\rceil$. In view of \cref{prop:killed-gap}, choose a fixed $B_p>2$ sufficiently large that
\[
 a_n(b_n)\le1/2,\qquad b_n=\lceil B_pn^2\rceil,
\]
for all sufficiently large $n$. Equations \eqref{eq:entropy-contract} then give
\begin{equation}\label{eq:entropy-half}
 H(\nu P_n^{b_n}\mid U_n)\le\tfrac12 H(\nu\mid U_n).
\end{equation}

We record the two elementary entropy facts used next. The log-sum inequality, obtained from convexity of $x\log x$, implies $H(\nu K\mid\pi K)\le H(\nu\mid\pi)$ for a Markov kernel. Thus $H_n(t)$ decreases. It also implies, after grouping the positive and negative parts of $\nu-\pi$ into two events, that entropy dominates the corresponding binary entropy. For fixed binary reference mass $y\in(0,1)$, the second derivative in $x$ of $x\log(x/y)+(1-x)\log((1-x)/(1-y))$ is $1/[x(1-x)]\ge4$. Its value and first derivative vanish at $x=y$. Integration twice gives binary entropy at least $2(x-y)^2$, and therefore Pinsker's inequality
\begin{equation}\label{eq:pinsker}
 \norm{\nu-\pi}_{\TV}^2\le H(\nu\mid\pi)/2.
\end{equation}
Boundary cases follow by continuity or have zero TV. These observations require no additional functional-inequality input.

Since $H_n(\tau_n)\le1$, \eqref{eq:pinsker} gives $d_n(\tau_n)\le1/\sqrt2$. By \cref{prop:lower},
\[
 \tau_n\ge t_{\rm mix}^{(n)}(1/\sqrt2)
 \ge t_n^\star-O_p(n^2).
\]
Conversely $H_n(0)=\log(n!)\le n\log n$, and \eqref{eq:entropy-contract} gives
\[
 H_n(m+s)\le C_p\rho_n^s\log(n!).
\]
At $s=\lceil(1+\delta)n^2\log n/\kappa\rceil$ this tends to zero for each fixed $\delta>0$, by \cref{prop:killed-gap}. These estimates prove both location bounds for $\tau_n$ in \cref{thm:entropy-main}; in particular $\tau_n/n^2\to\infty$.

By the definition of $\tau_n$, $H_n(\tau_n)\le1$ and $H_n(\tau_n-1)>1$ for large $n$. Iterating \eqref{eq:entropy-half} gives, for integers $j\ge0$,
\begin{equation}\label{eq:entropy-window-discrete}
 H_n(\tau_n+jb_n)\le2^{-j},\qquad
 H_n(\tau_n-1-jb_n)>2^j
\end{equation}
whenever the earlier time is nonnegative. The second inequality follows by applying the contraction forward from that earlier time to $\tau_n-1$ and rearranging. For a fixed real $s>0$, take $j=\lfloor\lceil sn^2\rceil/b_n\rfloor$ in the first bound and $j=\lfloor(\lfloor sn^2\rfloor-1)/b_n\rfloor$ in the second, replacing a negative $j$ by zero when necessary. Monotonicity fills the gaps between block endpoints. As $n\to\infty$ these choices are bounded below by a quantity growing linearly in $s$ minus a constant. Letting $s\to\infty$ proves both window limits in \cref{thm:entropy-main}.

This proof used only \cref{prop:lower,prop:killed-gap,lem:smoothing,lem:single-coupling} and the external entropy theorem. It does not depend on the two-sided static inequality. A sharp first-order entropy location, an entropy profile, and a sharp total-variation window remain distinct refinements.

\appendix
\section{Intermediate bounds and the change of strategy}\label{app:progression}
This appendix proves the intermediate bounds rather than using them as premises for the final theorem. Their role is to identify where successive powers of the deck size were lost. The main proof in Sections~\ref{sec:static}--\ref{sec:tvproof} does not depend on the intermediate exponents proved here.

\subsection{The initial factor-two upper bound}
Any two permutations can be joined by at most $n-1$ arbitrary transpositions. For example, a cycle of length $j$ is a product of $j-1$ transpositions, and one may sum over the disjoint cycles of their relative permutation. By \cref{lem:single-coupling} and the triangle inequality,
\[
 \norm{P_n^t(\sigma,\cdot)-P_n^t(\eta,\cdot)}_{\TV}
 \le(n-1)a_n(t),\qquad a_n(t)=\max_{z\in\cD}(T_n^t\1)(z).
\]
Averaging $\eta$ against $U_n$, which is stationary, gives $d_n(t)\le(n-1)a_n(t)$. The smoothing estimate in the proof of \cref{lem:flux} therefore gives, with $m=\lceil n^2\rceil$,
\[
 d_n(m+s)\le C_p n\rho_n^s.
\]
It follows from $-\log\rho_n\sim\kappa/n^2$ that the resulting upper leading coefficient is $1/\kappa$, twice the lower coefficient $1/(2\kappa)$.

More generally, if a static argument gives $C_n^\star\le n^{a+o(1)}$, then \cref{cor:adj-discrepancy} and the density identity \eqref{eq:B-density} give
\begin{equation}\label{eq:exponent-rule}
 \limsup_{n\to\infty}\frac{t_{\rm mix}^{(n)}(\eps)}{n^2\log n}\le\frac a\kappa.
\end{equation}
This is a consequence of a static exponent, not a claim that every such upper bound is sharp.

\subsection{One-scale erasure gives the exponent five-sixths}
\begin{proposition}\label{prop:five-sixths}
There is a universal $C$ such that $C_n^\star\le Cn^{5/6}$ for every $n\ge2$.
\end{proposition}
\begin{proof}
Choose consecutive blocks of sizes comparable to an integer $b$, and let $Q$ average over label and position permutations within those blocks. There are $R\asymp n/b$ blocks. The erasure estimate in \cref{lem:block-average} gives
\[
 \norm{f-Qf}_1\le Cb\cB_n(f),\qquad \cB_n(Qf)\le4\cB_n(f).
\]
The averaged function is table-measurable. Its smallest cell mean satisfies $\lambda_*\asymp b^2/n$. Provided $b^2/n\ge C_0\log n$, \cref{lem:coarse-table} gives
\[
 \norm{Qf-\E f}_1\le CR^2\sqrt n\,\cB_n(Qf)
 \le C\frac{n^{5/2}}{b^2}\cB_n(f).
\]
Therefore
\[
 C_n^\star\le C\left(b+\frac{n^{5/2}}{b^2}\right).
\]
Take $b\asymp n^{5/6}$, using a partition into block sizes differing by at most one. The two terms are then of order $n^{5/6}$, and the density hypothesis holds for large $n$. Increase $C$ for the remaining finite set of deck sizes.
\end{proof}
Equation~\eqref{eq:exponent-rule} turns this into the intermediate mixing coefficient $5/(6\kappa)$.

\subsection{Local merging and the exponent two-thirds}
The preceding coarse-table step routes exchanges across the entire deck. A local merging argument replaces that one global operation by a sequence of less expensive operations.

\begin{proposition}\label{prop:two-thirds}
For $b^2/n\ge K\log n$ and $1\le b\le n$, with $K$ a sufficiently large universal constant,
\begin{equation}\label{eq:multiscale-bound}
 C_n^\star\le C\left(b+\frac n{\sqrt b}+\frac{n^2\log n}{b^2}\right).
\end{equation}
In particular,
\[
 C_n^\star\le Cn^{2/3}(\log n)^{1/3}\qquad(n\ge2).
\]
\end{proposition}
\begin{proof}
We give the local estimate, sum it, and then remove its additive exceptional-event error.

\paragraph{A merging step.}
Suppose $F\in[0,1]$ depends only on a table whose current label and position blocks have sizes comparable to $\ell$. Merge neighboring label blocks, so the new label blocks have sizes comparable to $2\ell$, and let $Q_H$ perform the additional within-merged-block label averaging. The same proof will apply to position merging.

Set $\lambda_\ell=c_0\ell^2/n$, with a sufficiently small universal $c_0$, and let $G_\ell$ be the event that all current table entries are at least $\lambda_\ell$. Every cell has hypergeometric mean at least $2\lambda_\ell$. Thus \cref{lem:concentration} and a union bound over at most $n^2$ cells give
\begin{equation}\label{eq:merging-bad}
 \eta_\ell:=U_n(G_\ell^c)\le n^2e^{-c\ell^2/n}.
\end{equation}

For a level set $E_u=\{F>u\}$, let $R_{E_u,H}$ count crossing transpositions within merged label blocks. On $G_\ell$, any crossing table switch has at least $\lambda_\ell^2$ realizations by exchanging one label from each of two indicated cells. All these exchanges lie in the same merged label block. Hence $R_{E_u,H}=0$ or $R_{E_u,H}\ge\lambda_\ell^2$, so
\[
 \sqrt{R_{E_u,H}}\1_{G_\ell}\le\lambda_\ell^{-1}R_{E_u,H}.
\]
The total number of within-merged-block transpositions is at most $Cn\ell$. Applying the local square-root inequality \eqref{eq:sqrt-local}, whose gap parameter is of order $\ell$, gives
\[
 \norm{\1_{E_u}-Q_H\1_{E_u}}_1
 \le\frac C{\lambda_\ell\sqrt\ell}
 B_H^{\rm RT}(\1_{E_u})+C\sqrt n\,\eta_\ell.
\]
Here $B_H^{\rm RT}$ is the sum of the $L^1$ differences over the permitted complete transpositions. Integrating over $0\le u\le1$ and using Jensen for $F-Q_HF$ yields
\[
 \norm{F-Q_HF}_1
 \le\frac C{\lambda_\ell\sqrt\ell}B_H^{\rm RT}(F)
       +C\sqrt n\,\eta_\ell.
\]
Every permitted exchange has endpoints in a block of length $O(\ell)$. The path construction of \cref{lem:routing-basic}, restricted to each such block, has total congestion at most $C\ell^2$ at an adjacent edge. Therefore
\begin{equation}\label{eq:local-merge-cost}
 \norm{F-Q_HF}_1\le C\frac n{\sqrt\ell}B_L(F)+C\sqrt n\,\eta_\ell.
\end{equation}
The same estimate with $B_R$ holds for position merging.

\paragraph{Nested averaging.}
Start with consecutive blocks of sizes between constant multiples of $b$, and average an indicator $f=\1_E$ within both label and position blocks. This initial erasure costs at most $Cb\cB_n(f)$ by \cref{lem:block-average}. Then alternate label and position merging through scales $\ell_j\asymp2^jb$ until both partitions have one block. One construction is to start with a nearly balanced binary bisection tree of the interval $[n]$ and reverse its splits; at each level all lengths differ by a bounded factor. The intermediate label and position scales can differ by at most a factor two. This suffices for every constant in \eqref{eq:merging-bad}--\eqref{eq:local-merge-cost}.

A current function is always the \emph{direct} average of $f$ over the current label and position groups. Indeed, averaging over a coarser subgroup after a nested finer subgroup equals averaging over the coarser subgroup alone. The two sides commute. Thus \cref{lem:block-average} bounds each current boundary by $4\cB_n(f)$; there is no factor $4^j$ after $j$ levels. Summing the geometric costs gives
\[
 \sum_j\frac n{\sqrt{\ell_j}}\le C\frac n{\sqrt b}.
\]
The final average is the constant $\E f$. There are $O(\log n)$ merging operations and all $\ell_j\ge cb$, so the total additive error is bounded by
\[
 E_n(b)=C n^{5/2}\log n\,e^{-cb^2/n}.
\]
Consequently, for every set $E$,
\begin{equation}\label{eq:merging-indicator}
 2\alpha(1-\alpha)\le C\left(b+\frac n{\sqrt b}\right)\cB_n(\1_E)+E_n(b),
 \qquad\alpha=U_n(E).
\end{equation}
If $b$ is comparable to $n$, the same bound follows directly from the one-sided crude inequality, so possible terminal degenerate merges need not be used.

\paragraph{Rare sets.}
Complement $E$ to arrange $\alpha\le1/2$. For $\alpha\ge2E_n(b)$, the additive error in \eqref{eq:merging-indicator} is at most half the left side and is absorbed. If $\alpha<2E_n(b)$, then
\[
 \log(1/\alpha)\ge cb^2/n-\tfrac52\log n-\log\log n-O(1)
 \ge c'b^2/n
\]
under the assumed condition with $K$ sufficiently large. Applying \eqref{eq:smallset-path} to the full permutation multislice, with $N=n$ and $k=1$, gives
\[
 2\alpha(1-\alpha)\le2\alpha
 \le C\frac{n^2\log n}{b^2}\cB_n(\1_E).
\]
Together the two cases imply \eqref{eq:multiscale-bound} for indicators, and coarea gives arbitrary functions. Taking $b\asymp n^{2/3}(\log n)^{1/3}$ makes the first and third terms of that order and the middle term smaller. The hypothesis holds for large $n$; enlarge the final constant for the finite remainder.
\end{proof}

\subsection{Why recursive fibers improve on erasing blocks}
The terms $b$ and $n/\sqrt b$ in \eqref{eq:multiscale-bound} already balance at exponent $2/3$, even before accounting for rare sets. The initial $b$ is a one-sided erasure cost and cannot simply be replaced by the two-sided constant $C_b^\star$. For a uniform permutation of an even number $b$ of cards, let $E$ be the event that a specified card lies in the first half. Then
\[
 \norm{\1_E-1/2}_1=\frac12,\qquad B_R(\1_E)=\frac2b.
\]
Only the middle position edge contributes. Any one-sided constant valid for all such functions is therefore at least $b/4$.

The final proof retains rather than erases the internal permutations. The exact fiber decomposition \cref{lem:fiber-product} leaves independent smaller symmetric groups, and the exceptional-assignment inequality \cref{lem:exceptional} controls the long gaps needed to compare their generators. The resulting recursion \eqref{eq:static-recursion} is two-sided at the smaller permutation scale. Its proof is not an iteration of the one-sided erasure estimate, and this change is what permits the exponent $1/2+o(1)$.

\begin{center}
\begin{tabular}{@{}>{\raggedright\arraybackslash}p{.36\linewidth}>{\centering\arraybackslash}p{.25\linewidth}>{\centering\arraybackslash}p{.30\linewidth}@{}}
\toprule
Method & Static power, ignoring logarithms & TV upper coefficient in units of $1/\kappa$\\
\midrule
Direct discrepancy paths & $1$ & $1$\\
One-scale averaging & $5/6$ & $5/6$\\
Local multiscale merging & $2/3$ & $2/3$\\
Recursive conditional fibers & $1/2+o(1)$ & $1/2$\\
\bottomrule
\end{tabular}
\end{center}
The Wilson lower coefficient is $1/2$ in the same units throughout.

\subsection{A matching-order static lower bound}
\begin{proposition}\label{prop:static-lower}
The constants in \cref{thm:static} satisfy
\[
 \liminf_{n\to\infty}\frac{C_n^\star}{\sqrt n}
 \ge \sup_{c>0}\frac{1-e^{-c}}{2\sqrt c}>0.
\]
The supremum is approximately $0.319086$, attained at the positive solution of $e^c=1+2c$.
\end{proposition}
\begin{proof}
For $1\le m\le n/2$, let $E_m$ be the event that none of the labels $1,\ldots,m$ is in positions $1,\ldots,m$. Sampling without replacement gives
\[
 \alpha_{n,m}:=U_n(E_m)=\frac{\binom{n-m}{m}}{\binom nm}.
\]
Only the label transposition $s_m$ and the position transposition $s_m$ can change this event. Conditional on $E_m$, the probability that label $m+1$ lies in the first $m$ positions is $m/(n-m)$; this is the conditional exit probability for the label move. Similarly the probability that position $m+1$ contains a label at most $m$ is $m/(n-m)$, the exit probability for the position move. Since each transposition preserves $U_n$, its total expected indicator difference is twice its exit flow. Hence
\[
 \cB_n(\1_{E_m})=\frac{4m}{n-m}\alpha_{n,m},\qquad
 C_n^\star\ge\frac{(n-m)(1-\alpha_{n,m})}{2m}.
\]
Choose $m=\lfloor\sqrt{cn}\rfloor$ for fixed $c>0$. The product formula for $\alpha_{n,m}$ gives
\[
 \log\alpha_{n,m}
 =\sum_{j=0}^{m-1}\log\left(1-\frac m{n-j}\right)
 =-\frac{m^2}{n}+O\left(\frac{m^3}{n^2}\right)\longrightarrow-c.
\]
Division by $\sqrt n$ proves the displayed lower bound. Differentiating $(1-e^{-c})/(2\sqrt c)$ gives the equation $e^c=1+2c$. Convexity of the exponential, its derivative at zero, and its growth at infinity show that it has exactly one positive solution; the function tends to zero at both ends, so that solution is the maximum. The numerical approximation is not needed for the theorem.
\end{proof}

\section{Dependency and notation guide}\label{app:dependencies}
\subsection{External inputs and their precise uses}
The proof is self-contained relative to the following three stated, established functional-inequality results. We do not reproduce their original proofs.
\begin{enumerate}[leftmargin=*,label=(\roman*)]
\item The complete-graph specialization of the interchange-process spectral-gap theorem \cite{CaputoLiggettRichthammer2010}, in the normalization of \eqref{eq:external-poincare}. It is used for auxiliary complete-transposition Poincar\'e estimates and their products and quotients.
\item Salez's multislice log-Sobolev bound \cite{Salez2021}, in the normalization of \eqref{eq:external-LSI}. Its uses are the explicitly separated rare-set cases in the count, colored-word, coarse-table, and recursion arguments. No log-Sobolev statement about the overhand kernel itself is assumed.
\item The general Wasserstein-to-entropy contraction theorem of Caputo--M\"unch--Salez \cite{CaputoMunchSalez2025}, stated as \cref{thm:CMS}. Its reverse-kernel nonexpansiveness hypothesis is checked for $P_n^t$ at each integer scale at which it is used. It is needed only for the entropy theorem.
\end{enumerate}
The finite-grid Nash estimate, the square compactness and Poincar\'e arguments, the hypergeometric concentration estimates, coarea, product inequalities, and all model-specific couplings and counting comparisons are proved in the text. The analytic arguments use the elementary finite-dimensional spectral theorem, Fourier expansions on a square or torus, and the usual definitions of weak derivatives and Lebesgue integrals.

The logical separation of the main results is as follows. The lower bound depends on the exact one-card formula and Bernoulli variance tensorization. The pair-chain spectral and smoothing arguments imply the discrepancy estimates, without using the static inequality. The static inequality is a statement about uniform permutations proved without using the overhand chain. Combining those three independent parts gives the TV theorem. The entropy theorem uses the lower bound, pair estimates, and the external entropy-contraction theorem, but does not use the static inequality or the TV upper bound. Neither cutoff theorem uses the intermediate exponents in Appendix~\ref{app:progression}.

\subsection{Frequently used notation}
\renewcommand{\arraystretch}{1.12}
\begin{longtable}{@{}>{\raggedright\arraybackslash}p{.27\textwidth}>{\raggedright\arraybackslash}p{.67\textwidth}@{}}
\toprule
Notation & Meaning\\
\midrule\endhead
$\sigma(i)$; $s_i$ & Label at position $i$; the transposition $(i,i+1)$. Left multiplication exchanges labels, right multiplication exchanges positions.\\
$P_n$, $U_n$, $\mu_t$ & Packet-reversal kernel, uniform law on $\Sn$, and $\delta_{\id}P_n^t$.\\
$p,q,D,\kappa$ & Fixed cut probability, $1-p$, $q/p^2$, and $D\pi^2$.\\
$d_n(t)$, $H_n(t)$ & Total variation and relative entropy of $\mu_t$ from $U_n$.\\
$t_n^\star$, $\tau_n$ & The leading TV center $n^2\log n/(2\kappa)$ and the first integer time with $H_n(t)\le1$.\\
$K_n$, $f_k$, $\lambda_{n,k}$ & One-card kernel, its cosine modes, and their exact eigenvalues.\\
$\Phi$ & The full-deck cosine correlation used in the lower bound.\\
$\mathcal D_n$, $Q_n$, $T_n$ & Ordered unordered-pair state space $\{i<j\}$, its ordinary two-card kernel, and the killed discrepancy kernel.\\
$\rho_n$, $\beta_n$ & Spectral radius of $T_n$ and $1-\rho_n$. Entrywise nonnegativity is not positive semidefiniteness.\\
$\ell^2(\mathcal D_n)$ & Counting-measure norm, unlike the probability-measure norm $L^2(U_n)$.\\
$a_t$, $a_n(t)$ & The vector $T_n^t\1$ and its largest coordinate.\\
$B_L$, $B_R$, $\mathcal B_n$ & Left-adjacent boundary, right-adjacent boundary, and their sum.\\
$B_{\rm RT}$, $B_{\rm path}$ & Complete-transposition and adjacent-position $L^1$ boundaries on a multislice.\\
$C_n^\star$, $A(N)$ & Optimal two-sided centered-$L^1$ constant, and $\max_{1\le j\le\lfloor N\rfloor}C_j^\star$, with $C_1^\star=0$.\\
$M_{ab}$ & Number of labels from block $a$ occupying positions in block $b$.\\
$\omega$, $\Pi$ & The two families of assignment words, and the family of internal cell permutations, conditional on $M$.\\
$B_M(f)$, $b_{\rm int}(f;\omega)$ & Physical boundary within parent blocks on a fixed table fiber, and the conditional boundary of its cell permutations for fixed assignments.\\
$\Ent_\pi(F)$ & $\E_\pi(F\log F)-(\E_\pi F)\log\E_\pi F$ for $F\ge0$.\\
$d_{\rm tr}$, $W_{1,d}$, $K^*$ & Arbitrary-transposition metric, minimum expected metric distance over couplings, and reversed Markov kernel.\\
\bottomrule
\end{longtable}

\end{document}